\documentclass[11pt]{amsart}

\usepackage{amssymb,amsmath,amsthm}
\usepackage{graphicx}
\usepackage{color}
\usepackage{thmtools}
\usepackage{thm-restate}
\usepackage[shortlabels]{enumitem}
\usepackage{mathrsfs}
\usepackage[utf8]{inputenc}
\usepackage[T1]{fontenc}
\usepackage{dsfont}
\usepackage[left=3cm,right=3cm,top=3cm,bottom=3cm]{geometry}

\newtheorem{theorem}{Theorem}[section]
\newtheorem{lemma}[theorem]{Lemma}

\newtheorem{question}[theorem]{Question}
\newtheorem{corollary}[theorem]{Corollary}

\theoremstyle{definition}
\newtheorem{definition}[theorem]{Definition}

\theoremstyle{remark}

\numberwithin{equation}{section}

\newcommand{\uhr}{\upharpoonright}

\newcommand{\frakc}{\mathfrak{c}}

\newcommand{\frakp}{\mathfrak{p}}

\newcommand{\eps}{\varepsilon}

\newcommand{\R}{\mathbb{R}}

\newcommand{\E}{\mathbb{E}}

\newcommand{\aA}{\mathcal{A}}
\newcommand{\bB}{\mathcal{B}}
\newcommand{\cC}{\mathcal{C}}

\newcommand{\fF}{\mathcal{F}}

\newcommand{\jJ}{\mathcal{J}}

\newcommand{\concat}{{^\smallfrown}}

\DeclareMathOperator{\diam}{diam}
\DeclareMathOperator{\dom}{dom}

\DeclareMathOperator{\pr}{pr}

\DeclareMathOperator{\ran}{ran}

\newcommand{\rstr}{\restriction}

\newcommand{\sm}{\setminus}
\newcommand{\sub}{\subseteq}

\newcommand{\wh}{\widehat}

\newcommand{\Skin}[1]{\mathrm{Fin}\left(#1,2\right)}

\newcommand{\seq}[2]{\big\langle#1\colon\ #2\big\rangle}
\newcommand{\seqn}[1]{\big\langle#1\colon\ n\io\big\rangle}

\newcommand{\seqk}[1]{\big\langle#1\colon\ k\io\big\rangle}
\newcommand{\seql}[1]{\big\langle#1\colon\ l\io\big\rangle}

\newcommand{\seqi}[1]{\big\langle#1\colon\ i\io\big\rangle}

\newcommand{\seqnk}[1]{\big\langle#1\colon\ n,k\io\big\rangle}
\newcommand{\seqp}[1]{\big\langle#1\colon\ p\io\big\rangle}

\newcommand{\ctblsub}[1]{{\left[#1\right]^\omega}}
\newcommand{\finsub}[1]{{\left[#1\right]^{<\omega}}}

\newcommand{\iA}{\in\aA}

\newcommand{\w}{\omega}

\newcommand{\io}{\in\omega}

\newcommand{\bo}{{\beta\omega}}
\newcommand{\oo}{\omega^\omega}

\newcommand{\cso}{\ctblsub{\omega}}

\newcommand{\ioo}{\in\oo}

\newcommand{\Cantor}{2^\omega}

\begin{document}

\title[Minimally generated Boolean algebras]{Minimally generated Boolean algebras and the~Nikodym property}
\author[D.\ Sobota]{Damian Sobota}
\author[L. Zdomskyy]{Lyubomyr Zdomskyy}
\dedicatory{Dedicated to the memory of Kenneth Kunen (1943--2020)}

\address{Universit\"at Wien, Institut f\"ur Mathematik,
Kurt G\"odel Research Center, Kolingasse 14-16, 1090 Wien, Austria.}
\email{damian.sobota@univie.ac.at}
\urladdr{www.logic.univie.ac.at/~{}dsobota}
 \email{lzdomsky@gmail.com}
\urladdr{www.logic.univie.ac.at/~{}lzdomsky}
\thanks{The first author
was supported by the Austrian Science Fund FWF, Grants M 2500 and I
4570. The second author acknowledges the support by the same agency
through the Grants I 2374 and I 3709.}

\begin{abstract}
A Boolean algebra $\aA$ has \textit{the
Nikodym property} if every pointwise bounded sequence of bounded
finitely additive measures on $\aA$ is uniformly bounded.  Assuming
the Diamond Principle $\Diamond$, we will construct an example of a
minimally generated Boolean algebra $\aA$ with the Nikodym property.
The Stone space of such an algebra must necessarily be an Efimov
space. The converse is, however, not true---again under $\Diamond$
 we will provide an example of a minimally generated
Boolean algebra whose Stone space is Efimov but which does not
have the Nikodym property. The results have interesting
measure-theoretic and topological consequences.
\end{abstract}

\subjclass[2010]{Primary: 06E15, 28A33, 03E75. Secondary: 28E15, 03E35.}
\keywords{minimally generated Boolean algebras, Nikodym property, Efimov problem,  Grothendieck property, convergence of measures, inverse systems, Diamond principle}

\maketitle

\section{Introduction}

The celebrated \textit{Efimov problem} is a long-standing open question asking whether every infinite compact space contains either a non-trivial convergent sequence, or a copy of $\bo$, the Stone--\v{C}ech compactification of natural numbers. Although a full answer is not known, many consistent counterexamples to the problem, called \textit{Efimov spaces}, have been found, e.g. by Fedorchuk \cite{Fed75,Fed76,Fed77} or by Dow and his co-authors \cite{Dow05,DF07,DPM09,DS13}. Note here that a compact space $K$ contains a copy of $\bo$ if and only if it can be continuously mapped onto the product $[0,1]^\frakc$. Moreover, if $K$ can be continuously mapped onto the cube $[0,1]^\kappa$ for some cardinal number $\kappa\ge\omega$, then $K$ must necessarily carry a (Radon) probability measure of Maharam type $\kappa$ (see Definition \ref{def:maharam_type}). The converse is however not true---assuming the Diamond Principle $\Diamond$, D\v{z}amonja and Kunen \cite{DK93} constructed an Efimov space $K$ carrying a measure of Maharam type $\omega_1$. It follows that the question asking whether every infinite compact space either contains a non-trivial convergent sequence, or it carries a probability measure of uncountable Maharam type, is a strict weakening of the Efimov problem. However, it still does not hold true in ZFC, as D\v{z}amonja and Plebanek \cite{DP07} showed constructing an appropriate consistent counterexample---under the assumption of the Continuum Hypothesis they obtained an Efimov space carrying only probability measures of countable Maharam type.

In the same paper \cite{DP07}, D\v{z}amonja and Plebanek went even further and proved that every non-atomic probability measure on the limit of an inverse system of length at most $\omega_1$ and based on simple extensions of totally disconnected compact spaces is uniformly regular. The notion of the uniform regularity (Definition \ref{def:unif_reg_meas}) was introduced by Babiker \cite{Bab77} and trivially implies countable Maharam type. Inverse systems based on simple extensions are, intuitively speaking, such systems in which each successor space is obtained from its predecessor by splitting exactly one point (cf. Definition \ref{def:inv_sys_simp_ext}). Since the Efimov spaces obtained under $\Diamond$ in \cite{Fed76} or \cite{DK93} are limits of such inverse systems, it follows that an even stronger weakening of the Efimov problem, asking whether every compact space either contains a non-trivial convergent sequence or it carries a measure which is not uniformly regular, has consistently a negative answer.

Inverse systems based on simple extensions of compact spaces have a dual meaning in terms of so-called \textit{minimally generated} Boolean algebras, the notion introduced and studied by Koppelberg \cite{Kop88,Kop89}. Intuitively speaking, a minimally generated Boolean algebra is such an algebra which can be represented as an increasing chain of its proper subalgebras in such a way that there is no proper subalgebra between any two consecutive subalgebras in the chain (cf. Definition \ref{def:min_gen_ba}). Measure-theoretic properties of minimally generated Boolean algebras were investigated by Borodulin-Nadzieja \cite{PBN07}, who, i.a., showed that they can carry only measures of countable Maharam type. Forcing aspects of minimally generated Boolean algebras were deeply studied by Koszmider \cite{Kos99}.

In this paper we study further measure-theoretic properties of minimally generated Boolean algebras (or, equivalently, limits of inverse systems based on simple extensions). More precisely, assuming the Diamond Principle $\Diamond$ we prove in Theorems \ref{thm:min_gen_nik} and \ref{thm:min_gen_no_nik} the existence of two minimally generated Boolean algebras $\aA$ and $\bB$ such that $\aA$ has the Nikodym property whereas $\bB$ does not but its Stone space $St(\bB)$ is an Efimov space. The Nikodym property of Boolean algebras is defined as follows.

\begin{definition}\label{def:nikodym}
A Boolean algebra $\aA$ has \textit{the Nikodym property} if every sequence $\seqn{\mu_n}$ of finitely additive signed measures on $\aA$ which is pointwise bounded, i.e. $\sup_{n\io}\big|\mu_n(A)\big|<\infty$ for every $A\iA$, is uniformly bounded, i.e. $\sup_{n\io}\big\|\mu_n\big\|<\infty$.
\end{definition}

The property may look at first sight a bit strange, but by the virtue of the Riesz representation theorem for the duals of Banach $C(K)$-spaces it is closely related to the well-known Uniform Boundedness Principle (or, the Banach--Steinhaus theorem) of Banach spaces; it may also be expressed in terms of convergence of measures on totally disconnected compact spaces, see \cite[Prop. 2.4]{SZ19}. The class of Boolean algebras with the Nikodym property contains all $\sigma$-complete Boolean algebras (Nikodym \cite{Nik33}, Darst \cite{Dar67}; cf. also \cite{Sch82}, \cite{Hay81}, or \cite{Fre84_vhs}), but does not include those Boolean algebras whose Stone spaces contain non-trivial convergent sequences. The latter observation, together with the result of Koppelberg \cite{Kop89} and Borodulin-Nadzieja \cite{PBN07} stating that the Stone spaces of minimally generated Boolean algebras never contain copies of $\bo$, implies immediately that the Stone space $St(\aA)$ of the Boolean algebra $\aA$ constructed in Theorem \ref{thm:min_gen_nik} is an Efimov space. Since on the other hand the Boolean algebra of clopen subsets of the Efimov space obtained in Theorem \ref{thm:min_gen_no_nik} does not have the Nikodym property, one can say, in the spirit of the aforementioned weakenings of the Efimov problem, that the question whether every totally disconnected compact space either contains a non-trivial convergent sequence, or its Boolean algebra of clopen subsets has the Nikodym property, has a consistent negative answer\footnote{Let us only mention that both in the random model and the Laver model there exist totally disconnected Efimov spaces $K$ having an even stronger property: for every closed subset $L$ of $K$ the Boolean algebra of clopen subsets of $L$ does not have the Nikodym property.}.

The (consistent) existence of a minimally generated Boolean algebra with the Nikodym property, proved in Theorem \ref{thm:min_gen_nik}, has important consequences, since it helps to understand differences between the Nikodym property and the closely-related Grothendieck property, defined in the following way.

\begin{definition}\label{def:grothendieck}
A Boolean algebra $\aA$ has \textit{the Grothendieck property} if every sequence $\seqn{\mu_n}$ of signed Radon measures on the Stone space $St(\aA)$ which is weakly* convergent to $0$, i.e. $\lim_{n\to\infty}\int_{St(\aA)}fd\mu_n=0$ for every $f\in C(St(\aA))$, is weakly convergent to $0$, i.e. $\lim_{n\to\infty}\mu_n(B)=0$ for every Borel subset $B\sub St(\aA)$.
\end{definition}

The Grothendieck property is an important property, originating actually from the general Banach space theory, see Diestel \cite{Die73}. Similarly as in the case of the Nikodym property, the class of Boolean algebras with the Grothendieck property contains all $\sigma$-complete Boolean algebras (Grothendieck \cite{Gro53}; cf. again \cite{Sch82}, \cite{Hay81}, or \cite{Fre84_vhs}) and does not contain any Boolean algebra whose Stone space has a non-trivial convergent sequence. It does not contain also any minimally generated Boolean algebra (\cite[Cor. 9.11]{KSZ20}), which in the light of Theorem \ref{thm:min_gen_nik} shows that the both properties are essentially different. This is crucial, for it has always been a notorious problem to find any example of a Boolean algebra with only one of the named properties. The Boolean algebra $\jJ$ of all Jordan-measurable subsets of the unit interval $[0,1]$ is known to have the Nikodym property but not the Grothendieck property (Schachermayer \cite[Prop. 3.2--3.3]{Sch82}; cf. \cite{GW83}), but no reverse ZFC example has been found---the only known Boolean algebra with the Grothendieck property but without the Nikodym property was constructed by Talagrand \cite{Tal84} under the assumption of the Continuum Hypothesis. Let us note here that the Jordan algebra $\jJ$ is not minimally generated, as its Stone space contains many copies of $\bo$ (\cite[Prop 3.11]{Sch82}), thus our example from Theorem \ref{thm:min_gen_nik} seems to differ vitally from the previously known examples of Boolean algebras with the Nikodym property but without the Grothendieck property, studied e.g. in \cite{GW83}. Moreover, in Section \ref{sec:consequences} we list several measure-theoretic properties of minimally generated Boolean algebras which are never possessed by Boolean algebras with the Grothendieck property (such as carrying only measures of countable Maharam type)---it yields consequently that both of the properties may very much differ from the point of view of topological measure theory.

%

\section{Preliminaries}

Our notation is mostly standard and follows Kunen \cite{Kun80} (set theory), Bogachev \cite{Bog07} (measure theory) and Givant and Halmos \cite{GH09} (Boolean algebras).

By $\E$ we denote the set of even natural numbers, i.e. $\E=\{2n\colon n\io\}$.

If $I$ is a set, then by $\mathrm{Fin}(I,2)$ we denote the set of all finite $0-1$ sequences with domains in $I$, and by $\finsub{I}$ the family of all finite subsets of $I$. If $\alpha\le\beta$ are ordinals and $\sigma\in\Skin{\alpha}$, then $[\sigma]_\beta$ denotes the (clopen) subset of $2^\beta$ consisting of all $x$ such that $x(i)=\sigma(i)$ for every $i\in\dom(\sigma)$. Similarly, if $F\in\finsub{\Skin{\alpha}}$, then $[F]_\alpha=\bigcup_{\sigma\in F}[\sigma]_\alpha$. If $\alpha\le\beta$, then $\pr_\alpha^\beta\colon2^\beta\to 2^\alpha$ is the natural projection, i.e. $\pr_\alpha^\beta(x)=x\rstr\alpha$ for every $x\in 2^\beta$.

We make the assumption that all compact spaces considered in this paper are \textbf{Hausdorff}. A convergent sequence $\seqn{x_n}$ in a given compact space is \textit{non-trivial} if $x_m\neq\lim_{n\to\infty}x_n$ for every $m\io$ and $x_n\neq x_m$ for every $n\neq m\io$.

If $\aA$ is a Boolean algebra, then by $St(\aA)$ we denote its Stone space, i.e. the totally disconnected compact space of all ultrafilters on $\aA$. If $K$ is a totally disconnected compact space, then $Clopen(K)$ denotes its Boolean algebra of clopen subsets. Recall that $Clopen(St(\aA))$ is isomorphic to $\aA$ for every Boolean algebra $\aA$.

For a Boolean algebra $\aA$ we will denote its zero and unit elements by $0_\aA$ and $1_\aA$, respectively, and its meet and join operations by $\wedge$ and $\vee$. If $A\in\aA$, then its corresponding clopen subset of the Stone space $St(\aA)$ will be also denoted by $A$---this should not introduce any confusion.

All \textit{measures} considered on Boolean algebras are assumed to be finitely additive, signed and finite (bounded), i.e. if $\mu$ is a measure on a Boolean algebra $\aA$, then $\|\mu\|=\sup\big\{|\mu(A)|+|\mu(B)|\colon\ A,B\in\aA,\ A\wedge B=0_\aA\big\}<\infty$. If on the other hand $\mu$ is a measure on a compact space $K$, then it is assumed to be Radon, that is, $\mu$ is $\sigma$-additive, Borel, regular and finite (bounded), i.e. $\|\mu\|=\sup\big\{|\mu(A)|+|\mu(B)|\colon\ A,B\in Bor(K),\ A\cap B=\emptyset\big\}<\infty$. Recall that each measure $\mu$ on a Boolean algebra $\aA$ extends uniquely to a Radon measure on $St(\aA)$---we will denote this extension by $\wh{\mu}$. However for a clopen $V\sub St(\aA)$ and a measure $\mu$ on $\aA$ we will not distinguish between $\mu(A)$ and $\wh{\mu}(A)$ and usually write simply $\mu(A)$.

\subsection{Minimally generated Boolean algebras}

Minimally generated Boolean algebras are in a sense the simplest possible infinite Boolean algebras.

\begin{definition}\label{def:min_gen_ba}
Let $\bB$ be a Boolean algebra and $\aA$ its subalgebra. We say that $\bB$ is \textit{a minimal extension} of $\aA$ if there is no subalgebra $\cC$ of $\bB$ such that $\aA\subsetneq\cC\subsetneq\bB$.

We say that $\aA$ is \textit{minimally generated} if there is an increasing sequence $\seq{\aA_\alpha}{\alpha\le\delta}$ of subalgebras of $\aA$ such that
\begin{itemize}
    \item $\aA=\aA_\beta$,
    \item $\aA_0$ is isomorphic to the free countable Boolean algebra $Fr(\omega)$,
    \item $\aA_\gamma=\bigcup_{\alpha<\gamma}\aA_\alpha$ for every limit $\gamma\le\delta$, and
    \item $\aA_{\alpha+1}$ is a minimal extension of $\aA_\alpha$ for every $\alpha<\beta$.
\end{itemize}
\end{definition}

(Note that some authors, e.g. \cite{PBN07}, require that $\aA_0=\{0,1\}$.) For basic information concerning minimally generated Boolean algebras we refer the reader e.g. to \cite{Kop88}, \cite{Kop89}, and \cite{PBN07}. The notion of minimally generated Boolean algebras is dual to the following topological one (cf. \cite[Lemma 3.3]{PBN07}).

\begin{definition}\label{def:inv_sys_simp_ext}
An inverse system $\seq{K_\alpha,\pi^\beta_\alpha}{0\le\alpha<\beta\le\delta}$ of totally disconnected compact spaces is \textit{based on simple extensions} if
\begin{itemize}
    \item it is \textit{continuous}, i.e. for every limit ordinal $\gamma\le\delta$ the space $K_\gamma$ is the limit of the partial inverse system $\seq{K_\alpha,\pi^\beta_\alpha}{0\le\alpha<\beta\le\gamma}$,
    \item $K_0=2^\omega$,
    \item for every $\alpha<\delta$ the space $K_{\alpha+1}$ is \textit{a simple extension} of $K_\alpha$, i.e. there is $x_\alpha\in K_\alpha$ such that $\big|\big(\pi_\alpha^{\alpha+1}\big)^{-1}\big(x_\alpha\big)\big|=2$ and for every $y\in K_\alpha\sm\big\{x_\alpha\big\}$ it holds $\big|\big(\pi_\alpha^{\alpha+1}\big)^{-1}(y)\big|=1$.
\end{itemize}
\end{definition}

For general information concerning inverse systems of compact spaces, see \cite[Sections 2.5 and 3.2]{Eng89}. 

\begin{lemma}\label{lem:simple_ext_empty_int}
Let $\seq{K_\alpha,\pi^\beta_\alpha}{0\le\alpha<\beta\le\delta}$ be an inverse system based on simple extensions such that $K_\alpha$ is a perfect closed subset of $2^\alpha$ for every $\alpha\le\delta$. Fix $\gamma<\delta$. If $x_\gamma$ is the only point in $K_\gamma$ such that $\big|\big(\pi^{\gamma+1}_{\gamma}\big)^{-1}\big(x_\gamma\big)\big|=2$, then
    $\big(\pi^{\delta}_{\gamma+1}\big)^{-1}\big(x_\gamma^0\big)$ has empty interior in $K_{\delta}$ for every $x_\gamma^0\in\big(\pi^{\gamma+1}_{\gamma}\big)^{-1}\big(x_\gamma\big)$. 
\end{lemma}
\begin{proof}
    For the sake of contradiction assume that
    $\big(\pi^{\delta}_{\gamma+1}\big)^{-1}\big(x_\gamma^0\big)$ has non-empty
    interior in $K_{\delta}$ for some $x_\gamma^0\in\big(\pi^{\gamma+1}_{\gamma}\big)^{-1}\big(x_\gamma\big)$ and pick $\sigma\in\Skin{\delta}$ with the minimal
    possible $\max(\dom(\sigma))$ (which we will denote by $\eta$ in what
    follows) such that
    \[\emptyset\neq[\sigma]_{\delta}\cap K_{\delta}\sub\big(\pi^{\delta}_{\gamma+1}\big)^{-1}\big(x_\gamma^0\big),\]
    or equivalently
    \[\emptyset\neq[\sigma]_{\eta+1}\cap
    K_{\eta+1}\sub\big(\pi^{\eta+1}_{\gamma+1}\big)^{-1}\big(x_\gamma^0\big).\]
    Note that $\big(\pi^{\eta}_{\gamma+1}\big)^{-1}\big(x_\gamma^0\big)$ is
    nowhere dense in $K_\eta$, as otherwise there would exist
    $\tau\in\Skin{\eta}$ with
    \[\emptyset\neq[\tau]_\eta\cap K_\eta\sub
    \big(\pi^{\eta}_{\gamma+1}\big)^{-1}\big(x_\gamma^0\big),\] and hence
    \[\emptyset\neq [\tau]_{\delta}\cap K_{\delta}\sub
    \big(\pi^{\delta}_{\gamma+1}\big)^{-1}\big(x_\gamma^0\big),\] thus
    contradicting the minimality of $\eta$. 
    Let
    $x_\eta$ be the only element of $K_\eta$ with
    $\big|\big(\pi^{\eta+1}_{\eta}\big)^{-1}\big(x_\eta\big)\big|=2$. Note that the function
    \[\pi^{\eta+1}_{\eta}\rstr\Big( K_{\eta+1}\sm
    \big(\pi^{\eta+1}_{\eta}\big)^{-1}\big(x_\eta\big)\Big)\]
    is a homeomorphism between  open dense subsets $\Big( K_{\eta+1}\sm
    \big(\pi^{\eta+1}_{\eta}\big)^{-1}\big(x_\eta\big)\Big)$ of $K_{\eta+1}$ and
    $K_\eta\setminus\{x_\eta\}$ of
    $K_\eta$,
    respectively, thus making it impossible for any subset $Z\sub K_\eta$ that
    $\big(\pi^{\eta+1}_\eta\big)^{-1}[Z]$ has non-empty interior in $K_{\eta+1}$
    while $Z$ has empty interior in $K_\eta$. However, it follows from the
    above that  this is the case for
    $Z=\big(\pi^{\eta}_{\gamma+1}\big)^{-1}\big(x_\gamma^0\big)$, which leads to
    a contradiction.
%
\end{proof}

The following two lemmas show how to extend a given compact space (a Boolean algebra) in a simple (minimal) way. Lemma \ref{lem:simple_ext_convergent_antichain} will be used in the proof of Theorem \ref{thm:min_gen_nik}, whereas Lemma \ref{lem:min_ext_convergent_antichain} will be necessary for the proof of Theorem \ref{thm:min_gen_no_nik}. The proofs are folklore or easy (for the proof of the first part of Lemma \ref{lem:min_ext_convergent_antichain} see \cite[Lemma 3.14]{PBN07}).

\begin{lemma}\label{lem:simple_ext_convergent_antichain}
Let $K$ be a totally disconnected compact space and $\seqn{V_n}$ an antichain of clopens in $K$ converging to some point $t\in K$. Fix an infinite co-infinite subset $A\sub\omega$ and define the following closed subset of $K\times\{0,1\}$:
\[L=\Big(\bigcup_{n\in A}V_n\cup\{t\}\Big)\times\{0\}\ \cup\ \Big(K\sm\bigcup_{n\in A}V_n\Big)\times\{1\}.\]
Then, $L$, together with the natural projection, is a simple extension of $K$ and the set
\[\Big(\bigcup_{n\in A}V_n\cup\{t\}\Big)\times\{0\}\]
is a clopen subset of $L$.
\end{lemma}

\begin{lemma}\label{lem:min_ext_convergent_antichain}
Let $\aA$ be a Boolean algebra and $\seqn{V_n}$ an antichain in $\aA$. Assume that in $St(\aA)$ the clopen sets $V_n$'s converge to some point (ultrafilter) $t\in St(\aA)$. Fix an infinite co-infinite subset $A\sub\omega$. Let $\cC$ be a Boolean algebra containing $\aA$ and such that the supremum $V=\bigvee_{n\in A}V_n$ exists in $\cC$. Then, the following two statements hold.
\begin{enumerate}
    \item The Boolean algebra $\bB$ generated by $\aA$ and $V$ (in $\cC$) is a minimal extension of $\aA$ and the ultrafilter $t$ is the only ultrafilter in $\aA$ that extends to two different ultrafilters in $\bB$;
    \item Let $K$ be a compact space and $\varphi\colon K\to St(\aA)$ a homeomorphism. Put:
    \[L=\Big(\bigcup_{n\in A}\varphi^{-1}\big[V_n\big]\cup\big\{\varphi^{-1}(t)\big\}\Big)\times\{0\}\ \cup\ \Big(K\sm\bigcup_{n\in A}\varphi^{-1}\big[V_n\big]\Big)\times\{1\}\]
    and define the mapping $\varphi'\colon L\to St(\bB)$ as follows: $\varphi'(x,y)$ is the ultrafilter $\varphi(x)$ extended to the algebra $\bB$ by the element $V$ if $y=0$ and by the element $V^c$ if $y=1$. Then, $\varphi'$ is a homeomorphism and satisfies the equality $\rho\circ\varphi'=\varphi\circ\pi$, where $\rho$ is the restriction of ultrafilters in $\bB$ to the subalgebra $\aA$, i.e. $\rho(x)=x\cap\aA$ for every $x\in St(\bB)$, and $\pi$ is the natural projection from $K\times\{0,1\}$ onto $K$. It also holds:
    \[(\varphi')^{-1}[V]=\Big(\bigcup_{n\in A}\varphi^{-1}\big[V_n\big]\cup\big\{\varphi^{-1}(t)\big\}\Big)\times\{0\}.\]
\end{enumerate}
\end{lemma}

Note that Lemma \ref{lem:simple_ext_convergent_antichain} implies that the space $L$ in Lemma \ref{lem:min_ext_convergent_antichain}.(2) is a simple extension of the space $K$.

\subsection{Anti-Nikodym sequences of measure}

We start with the following basic definition concerning the failure of the Nikodym property.

\begin{definition}
A sequence $\seqn{\mu_n}$ of measures on a Boolean algebra $\aA$ is \textit{anti-Nikodym} if it is pointwise bounded but not uniformly bounded.
\end{definition}

Obviously, a Boolean algebra $\aA$ has the Nikodym property if and only if there are no anti-Nikodym sequences on $\aA$; for more information, see \cite[Section 4]{Sob19}. The following lemma shows a crucial property of anti-Nikodym sequences.

\begin{lemma}\label{lem:aN_antichain}
Let $K$ be a totally disconnected compact space and $\seqn{\mu_n}$ an anti-Nikodym sequence of measures on $Clopen(K)$. Then, there is an antichain $\seqk{A_k}$ of clopen subsets of $K$ and a sequence $\seqk{n_k\io}$ such that for every $k\io$ we have $\big|\mu_{n_k}\big(A_k\big)\big|\ge2^k\cdot k$.
\end{lemma}
\begin{proof}
See \cite[Lemma 4.4]{Sob19} and use \cite[Lemma 4.7]{Sob19} inductively.
\end{proof}

Recall that a family $\seq{A_\alpha}{\alpha<\omega_1}$ of infinite subsets of $\omega$ is \textit{almost disjoint} if $A_\alpha\cap A_\beta$ is finite for every $\alpha\neq\beta<\omega_1$.

\begin{lemma}\label{lem:measures_ad_families}
Let $\aA$ be an infinite Boolean algebra, $\mu$ a measure on $\aA$ and  $\seqn{B_n}$ an antichain in $\aA$. Fix an almost disjoint family $\seq{A_\alpha\in\cso}{\alpha<\omega_1}$. For each $\zeta<\omega_1$ and $i\in A_\zeta$ let $C_i^\zeta\in\aA$ be such that $C_i^\zeta\leq B_i$, and assume that the supremum $S_\zeta=\bigvee\big\{C_i^\zeta\colon\ i\in A_\zeta\big\}$ exists in $\aA$ for every $\zeta<\omega_1$. Then, there is $\eta<\omega_1$ such that
\[\big|\wh{\mu}\big|\big(\partial_{St(\aA)}\Big(\bigcup_{i\in A_\zeta}C_i^\zeta\Big)\big)=0\]
for every $\zeta\in[\eta,\omega_1)$.
\end{lemma}
\begin{proof}
For every $\zeta<\omega_1$, let
\[D_\zeta=\partial_{St(\aA)}\Big(\bigcup_{i\in A_\zeta}C_i^\zeta\Big).\]
Then, for every $\zeta<\xi<\omega_1$ we have $D_\zeta\cap D_\xi=\emptyset$. To see this, assume that there is $x\in D_\zeta\cap D_\xi$. Of course, $x\in S_\zeta$, but  $x\not\in\bigcup_{i\in A_\zeta}C_i^\zeta$, since each $C_i^\zeta$ is open in $St(\aA)$. Similarly, $x\in S_\xi$ and $x\not\in\bigcup_{i\in A_\xi}C_i^\xi$. It follows that $x\not\in\bigcup_{i\io}B_i$, since also $B_i\cap C_j^\zeta=B_i\cap C_j^\xi=\emptyset$ for every $i\neq j\io$. Let
\[T=S_\zeta\sm\bigcup_{i\in A_\zeta\cap A_\xi}B_i;\]
then, $T$ is a clopen in $St(\aA)$ and $x\in T$. But for every $i\in A_\xi$ we also have $T\cap C_i^\xi=\emptyset$ (or even $T\cap B_i=\emptyset$), so $x\not\in D_\xi$, which is a contradiction.

It follows that the collection $\big\{D_\zeta\colon\ \zeta<\omega_1\big\}$ consists of pairwise disjoint closed subsets of $St(\aA)$. Since $\big|\wh{\mu}\big|$ is bounded, there exists $\eta<\omega_1$ such that for every $\zeta\in[\eta,\omega_1)$ we have $\big|\wh{\mu}\big|\big(D_\zeta\big)=0$.
\end{proof}

Note that in the above lemma the elements $C_i^\zeta$'s may be zero.

\begin{lemma}\label{lem:antichain_tails}
    Let $\seql{\mu_l}$ be a sequence of measures on a  Boolean algebra
    $\aA$ and $\seql{V_l}$ an antichain in $\aA$. Then, there exists a
    strictly increasing sequence $\seqp{l_p\io}$ such that
    \[\sum_{i>p}\big|\mu_{l_p}\big|\big(V_{l_i}\big)<1\]
    for every $p\io$.
\end{lemma}
\begin{proof}
Let $l_0=0$. Since the measure $\mu_{l_0}$ is finite and the sequence $\seqi{V_i}$ is an antichain in $\aA$, there is $l_1>l_0$ such that
\[\sum_{i\ge l_1}\big|\mu_{l_0}\big|\big(V_i\big)<1.\]
Similarly, since $\mu_{l_1}$ is finite, there is $l_2>l_1$ such that
\[\sum_{i\ge l_2}\big|\mu_{l_1}\big|\big(V_i\big)<1.\]
Continue in this manner to obtain a sequence $\seqp{l_p\io}$ with the required properties.
\end{proof}

We finish this section with the following simple technical lemma.

\begin{lemma}\label{lem:antichain_witness_selection}
    Let $\seqnk{m_{nk}\in\R}$ be an infinite square matrix satisfying the following conditions:
    \begin{itemize}
        \item $m_{nk}\ge0$ for every $n,k\io$,
        \item $\sum_{k\io}m_{nk}<\infty$ for every $n\io$,
        \item $\sup_{n\io}m_{nk}<\infty$ for every $k\io$,
        \item $m_{nn}>n$ for every $n\io$.
    \end{itemize}
    Then, there is a sequence $\seqp{l_p\io}$ such that
    \[m_{l_pl_p}>\sum_{i=0}^{p-1}m_{l_pl_i}+p+1\]
    for every $p\io$.
\end{lemma}
\begin{proof}
    We construct the sequence $\seqp{l_p}$ inductively. First, since $m_{11}>1$, let $l_0=1$---the conclusion for $p=0$ trivially holds. Let now $r\ge0$ and assume that we have constructed a sequence $l_0,\ldots,l_r$ as required. Let:
    \[\alpha=\sum_{i=0}^r\big(\sup_{n\io}m_{nl_i}\big)+r+2;\]
    then, for $l_{r+1}=\lceil\alpha\rceil$ (the ceiling of $\alpha$) the thesis holds:
    \[m_{l_{r+1}l_{r+1}}>l_{r+1}\ge\alpha\ge\sum_{i=0}^rm_{l_{r+1}l_i}+r+2.\]
\end{proof}

\subsection{The Diamond Principle $\Diamond$}

There are many equivalent formulations of the Diamond Principle $\Diamond$, see e.g. Devlin \cite{Dev79}---in this paper we use the following one (cf. D\v{z}amonja and Plebanek \cite[p. 2071]{DP07}):
\begin{equation*}
\tag{$\Diamond$}\parbox{\dimexpr\linewidth-7em}{there exists an $\omega_1$-sequence $\seq{f_\alpha\in\big(2^\alpha\big)^\omega}{\omega\le\alpha<\omega_1}$ such that for every $f\in\big(2^{\omega_1}\big)^\omega$ the set
    \[\Big\{\alpha\in\big[\omega,\omega_1\big)\colon\ (\forall n\io)\big(f(n)\rstr\alpha=f_\alpha(n)\big)\Big\}\]
    is stationary in $\omega_1$.}
\end{equation*}
Given any bijection $\theta:\w_1\to\w\times\w_1,$ there is a club
$C\sub\w_1$ such that $\theta[\alpha]=\w\times\alpha$ for all
$\alpha\in C$. Thus the Diamond Principle is equivalent to
\begin{equation*}
\tag{$\Diamond'$}\parbox{\dimexpr\linewidth-7em}{there exists an
    $\omega_1$-sequence
    $\seq{f_\alpha\in\big(2^{\w\times\alpha}\big)^\omega}{\omega\le\alpha<\omega_1}$
    such that for every $f\in\big(2^{\w\times\omega_1}\big)^\omega$ the
    set
    \[\Big\{\alpha\in\big[\omega,\omega_1\big)\colon\ (\forall n\io)\big(f(n)\rstr(\w\times\alpha)=f_\alpha(n)\big)\Big\}\]
    is stationary in $\omega_1$.}
\end{equation*}
Of course, since 
$\Big(2^{\omega\times\alpha}\Big)^\omega=\Big(\big(2^\omega\big)^\alpha\Big)^\omega$
for every $\alpha\ge\omega$ and $\big|2^\omega\big|=\big|\R\big|$,
($\Diamond'$) may be reformulated as follows:
\begin{equation}
\tag{$\Diamond^*$}\parbox{\dimexpr\linewidth-7em}{there exists a sequence
    $\seq{h_\alpha\in\big(\R^\alpha\big)^\omega}{\omega\le\alpha<\omega_1}$
    such that for every $h\in\big(\R^{\omega_1}\big)^\omega$ the set
    \[\Big\{\alpha\in\big[\omega,\omega_1\big)\colon\ (\forall n\io)\big(h(n)\rstr\alpha=h_\alpha(n)\big)\Big\}\]
    is stationary in $\omega_1$.}
\end{equation}

We need to adjust ($\Diamond^*$) to a slightly more measure-theoretic
setting. The following fact is straightforward (cf. the proof of Lemma \ref{lem:var_det_club}).

\begin{lemma}\label{lem:club}
    Let $\seq{F_\xi}{\xi<\omega_1}$ be an enumeration of all elements of
    $\finsub{\Skin{\w_1}}$. Then, the set
    \[C=\Big\{\alpha<\omega_1\colon\ \big\{F_\xi\colon\ \xi<\alpha\big\}=\finsub{\Skin{\alpha}}\Big\}\]
    is a club in $\omega_1$.
\end{lemma}

Using Lemma \ref{lem:club}, we derive the following consequence of
($\Diamond^*$). As we will  see in the proof of Theorem
\ref{thm:min_gen_nik}, the sets $\finsub{\Skin{\alpha}}$ for
$\alpha\le\omega_1$ will allow us to describe in a natural way the
Boolean algebras $Clopen\big(2^\alpha\big)$.
\begin{lemma}\label{lem:diamond_clopens}
    Assuming $\Diamond$, there exist an $\omega_1$-sequence
    \[\seq{f_\alpha\in\big(\mathbb{R}^{[\Skin{\alpha}]^{<\omega}}\big)^\omega}{\omega\le\alpha<\omega_1}\]
    such that for every
    $f\in\big(\mathbb{R}^{[\Skin{\w_1}]^{<\omega}}\big)^\omega$ the
    set
    \[\Big\{\alpha\in\big[\omega,\omega_1\big)\colon\ (\forall n\io)\big(f(n)\rstr \finsub{\Skin{\alpha}}=f_\alpha(n)\big)\Big\}\]
    is stationary in $\omega_1$.
\end{lemma}
\begin{proof}
    Let $\seq{h_\alpha}{\alpha<\w_1\rangle}$ be a witness for ($\Diamond^*$), $\seq{F_\xi}{\xi<\omega_1}$ an enumeration of all elements of $\finsub{\Skin{\omega_1}}$ and
    $C$  the corresponding club given by Lemma
    \ref{lem:club}. For every $\alpha\in C$, $n\in\w$, and $F\in
    \finsub{\Skin{\alpha}}$ set  $f_\alpha(n)(F)=h_{\alpha}(n)(\xi)$,
    where $\xi$ is such that $F=F_\xi$. Such  $\xi$ exists because
    $\alpha\in C$. For every $\alpha\not\in C$ pick any $f_\alpha$ with
    appropriate domain and range.

    Fix any
    $f\in\big(\mathbb{R}^{[\Skin{\w_1}]^{<\omega}}\big)^\omega$,
    $n\in\w$, $\xi<\w_1$, and set $h(n)(\xi)=f(n)(F_\xi)$. Let $S$ be
    the stationary set witnessing ($\Diamond^*$) for $h$ and let $\alpha\in C\cap S$.
    Then for every $\xi<\alpha$ and $n\in\w$ we have
    $$ f(n)(F_\xi)=h(n)(\xi)=h_\alpha(n)(\xi)=f_\alpha(n)(F_\xi), $$
    which means $f(n)\rstr\finsub{\Skin{\alpha}}=f_\alpha(n)$ and thus
    completes our proof.
\end{proof}

For a function $f\in\R^{[\Skin{\alpha}]^{<\w}}$,
$\alpha\le\omega_1$, we define its \textit{total
    variation}\footnote{The index $\alpha$ is of course superfluous in
    this notation as $\alpha=\bigcup_{F\in\mathrm{dom}(f)}\bigcup_{s\in
        F}\mathrm{dom}(s)$. However, we believe that adding this index makes
    the corresponding formulas easier to understand.}
\[|f|_\alpha\in\big(\R\cup\{\infty\}\big)^{[\Skin{\alpha}]^{<\w}}\]
as follows:
\[|f|_\alpha(F)=\]
\[\sup\big\{|f(G)|+|f(H)|\colon\ G,H\in\finsub{\Skin{\alpha}}, [G]_\alpha\cap[H]_\alpha=\emptyset, [G]_\alpha\cup[H]_\alpha\sub[F]_\alpha\big\}\]
\smallskip

\noindent for every $F\in\finsub{\Skin{\alpha}}$. If $\beta\le\alpha$ and $F\in\finsub{\Skin{\beta}]}$,
then we say that $|f|_\alpha(F)$ is \textit{$\beta$-determined}  if
\[|f|_\alpha(F)=\big|f\uhr\finsub{\Skin{\beta}}\big|_\beta(F).\]
The total variation $|f|_\alpha$ is \textit{$\beta$-determined}  if
$|f|_\alpha(F)$ is $\beta$-determined  for every $F\in
\finsub{\Skin{\beta}}$. Note that in order to determine
$|f|_\alpha(F)$ for some $F\in\finsub{\Skin{\alpha}}$, we only need
countably many sets from $\finsub{\Skin{\alpha}}$.

\begin{lemma}\label{lem:var_det_club}
    Let $f\in\R^{[\Skin{\w_1}]^{<\w}}$. Then, the set
    \[D=\Big\{\alpha<\omega_1\colon\ |f|_{\w_1} \text{ is $\alpha$-determined }\Big\}\]
    is a club in $\omega_1$.
\end{lemma}
\begin{proof}
%
    Let $\seqn{\alpha_n}$ be a strictly increasing sequence  in $D$ and
    $\alpha=\sup_{n\io}\alpha_n$. For every $F\in
    \finsub{\Skin{\alpha}}$ there is $n\io$ such that $F\in
    \finsub{\Skin{\alpha_n}}$, so $|f|_{\w_1}(F)$ is
    $\alpha_n$-determined  and hence also $\alpha$-determined. It
    follows that $\alpha\in D$ and thus $D$ is closed.

    Let now $\beta<\omega_1$. Let  $\seqn{\alpha_n<\omega_1}$ be such a
    sequence that $\alpha_0=\beta$ and $|f|_{\w_1}(F)$ is
    $\alpha_{n+1}$-determined for all $F\in\finsub{\Skin{\alpha_n}}$
    and $n\in\w$. Put $\alpha=\sup_{n\io}\alpha_n$. Of course,
    $\alpha\ge\beta$. It follows that $|f|_{\w_1} $ is
    $\alpha$-determined. Indeed, if $F\in\finsub{\Skin{\alpha}}$, then
    $F\in\finsub{\Skin{\alpha_n}}$ for some $n\io$, so $|f|_{\w_1}(F)$
    is $\alpha_{n+1}$-determined and hence also $\alpha$-determined.
    This yields that $\alpha\in D$ and thus  $D$ is unbounded.
\end{proof}

\section{The first example}

We are ready to prove the main result of this paper.

\begin{theorem}\label{thm:min_gen_nik}
Assuming $\Diamond$, there exists a minimally generated Boolean algebra $\aA$ with the Nikodym property.
\end{theorem}
\begin{proof}
Let us fix the following sequences:
\begin{itemize}
    \item $\seq{f_\alpha\in\big(\mathbb{R}^{\mathrm{Fin}(\alpha,2)}\big)^\omega}{\omega\le\alpha<\omega_1}$,
    a sequence from Lemma \ref{lem:diamond_clopens};
    \item $\seq{A_\alpha}{\alpha<\omega_1}$, an almost disjoint family of subsets of $\cso$;
    \item $\seq{\seqn{F_n^\alpha\in\finsub{\Skin{\w_1}}}}{\alpha<\omega_1}$, an enumeration of all
    sequences of finite collections of finite partial functions from $\w_1$ to $2$ such that $\big[F_n^\alpha\big]_{\w_1}\cap\big[F_k^\alpha\big]_{\w_1}=\emptyset$ for every $\alpha<\w_1$ and $n\neq k\io$; and
    \item $\seq{s_\alpha\ioo}{\alpha<\omega_1}$, an enumeration of all strictly increasing sequences of natural numbers.
\end{itemize}
We will construct an inverse  system
$\seq{K_\alpha,\pi_\alpha^\beta}{\omega\le\alpha<\beta\le\omega_1}$
based on simple extensions and such that $K_\alpha$ will be a closed
perfect subset of the space $2^\alpha$ for every
$\omega\le\alpha\le\omega_1$, $\pi_\alpha^\beta\colon K_\beta\to
K_\alpha$ will be such that $\pi_\alpha^\beta=\pr_\alpha^\beta\rstr
K_\beta$ for every $\omega\le\alpha<\beta\le\omega_1$,
and the Boolean algebra $Clopen\big(K_{\omega_1}\big)$ will have the
Nikodym property. We start with $K_\omega=2^\omega$ and let us
assume that for some $\omega<\delta<\omega_1$ we have already
constructed an initial segment of our system
$\seq{K_\alpha,\pi_\alpha^\beta}{\omega\le\alpha<\beta<\delta}$.
If $\delta$ is a limit ordinal, then simply put
$K_\delta=\varprojlim\seq{K_\alpha}{\omega\le\alpha<\delta}$. If
$\delta=\gamma+1$ for some ordinal $\gamma$, then we proceed as
follows.
\medskip

We put $K_{\gamma+1}=K_\gamma\times\{0\}$ and let $\pi_\gamma^\delta\colon K_\delta\to K_\gamma$ to be simply a restriction of $\pr_\gamma^\delta$, unless all of the following
conditions simultaneously hold:
\begin{itemize}
\item[$(i)$]
For every $n\io$ and every distinct $F,G\in\finsub{\Skin{\gamma}}$
such that $[F]_\gamma\cap K_\gamma=[G]_\gamma\cap K_\gamma$ we have
$f_\gamma(n)(F)=f_\gamma(n)(G)$;
 \item[$(ii)$] For every $n\in\w$ the function
  $\mu^\gamma_n$ is  a measure on the Boolean algebra $Clopen\big(K_\gamma\big)$,
  where $\mu^\gamma_n(U)=f_\gamma(n)(F)$ for some (equivalently any)
  $F\in\finsub{\Skin{\gamma}}$ such that $[F]_\gamma\cap
  K_\gamma=U$; and
    \item[$(iii)$]  $\vec{\mu}_\gamma=\langle \mu^\gamma_n:n\in\w\rangle$ is  an
    anti-Nikodym sequence of measures on $Clopen\big(K_\gamma\big)$.
\end{itemize}
If the three conditions listed above hold, we shall make sure that
$\vec{\mu}_\gamma$ cannot be extended to an
    anti-Nikodym sequence of measures on $Clopen\big(K_{\w_1}\big)$
    by constructing $K_{\gamma+1}$ as described below.
First, notice that $K_\gamma$ is an infinite second countable
compact space with no isolated points, so it is homeomorphic to the
Cantor space $\Cantor$. If $\varphi_\gamma\colon K_\gamma\to\Cantor$
is a homeomorphism and $d$ denotes the standard metric on $\Cantor$,
then let $d_\gamma$ denote the metric on $K_\gamma$ defined as
 $d_\gamma(x,y)=d(\varphi_\gamma(x),\varphi_\gamma(y))$ for every $x,y\in K_\gamma$. This way,
  for fixed $n\io$, the set of preimages
  $F_n=\big\{\varphi_\gamma^{-1}\big[[\sigma]_\omega\big]\colon\ \sigma\in 2^n\big\}$ is a
  decomposition of $K_\gamma$ into $2^n$ disjoint clopen subsets of $d_\gamma$-diameter
  $1/2^n$.

Now, by Lemma \ref{lem:aN_antichain} there are a strictly increasing
sequence $\seqk{n_k\io}$ and an antichain $\seqk{W_k}$ of
clopen subsets of $K_\gamma$ satisfying for every $k\io$ the
following inequality:
\[\tag{1}\big|\mu^\gamma_{n_k}\big(W_k\big)\big|>2^k\cdot k.\]
 Let  $\beta_\gamma$ be the minimal ordinal $\beta<\w_1$ such that
 $F^\beta_k\in\finsub{\Skin{\gamma}}$ for all $k\in\w$, and for
the antichain $\seqk{W_k^\gamma}$ defined for every $k\io$ as $W_k^\gamma=\big[F^\beta_k\big]_\gamma\cap
K_\gamma$ there exists a strictly increasing sequence
$\seqk{n_k\io}$ such that  $(1)$ holds.
 Now let
$\xi_\gamma<\omega_1$ be the minimal ordinal such that the
inequality $(1)$ holds for the sequence
$\seqk{s_{\xi_\gamma}(k)}$ of natural numbers and the antichain
$\seqk{W_k^\gamma}$, i.e.
for every $k\io$ we have:
\[\tag{2}\big|\mu^\gamma_{s_{\xi_\gamma}(k)}\big(W_k^\gamma\big)\big|>2^k\cdot k.\]
By the triangle inequality, for each $k\io$ there is $\sigma_k\in
2^k$ such that
\[\big|\mu^\gamma_{s_{\xi_\gamma}(k)}\Big(W_k^\gamma\cap\varphi_\gamma^{-1}\big[[\sigma_k]_\omega\big]\Big)\big|>k.\]
Set
$V_k^\gamma=W_k^\gamma\cap\varphi_\gamma^{-1}\big[[\sigma_k]_\omega\big]$,
so  $V_k^\gamma$ is a clopen set in $K_\gamma$.
Since
$\diam_{d_\gamma}\big(V_k^\gamma\big)\le1/2^k$ for each $k\io$ and the
set $A_\gamma$ is infinite, there is a subsequence $\seql{k_l}$
contained in $A_\gamma$ for which the clopen sets $V_{k_l}^\gamma$
converge to some point $t_\gamma\in K_\gamma$ as $l\to\infty$. Since
$\vec{\mu}_\gamma$ is pointwise convergent and thus pointwise
bounded, by Lemma \ref{lem:antichain_witness_selection} (applied to
$m_{ij}=\big|\mu^\gamma_{s_{\xi_\gamma}(k_i)}\big(V_{k_j}^\gamma\big)\big|$,
$i,j\io$) and Lemma \ref{lem:antichain_tails}, there is a sequence
$\seqp{l_p}$ such that
\[\tag{3}\Big|\mu^\gamma_{s_{\xi_\gamma}(k_{l_p})}\Big(V_{k_{l_p}}^\gamma\Big)\Big|>\sum_{i=0}^{p-1}\Big|\mu^\gamma_{s_{\xi_\gamma}(k_{l_p})}\Big(V_{k_{l_i}}^\gamma\Big)\Big|+p+1\]
and
\[\tag{4}\sum_{i>p}\Big|\mu^\gamma_{s_{\xi_\gamma}(k_{l_p})}\Big|\Big(V_{k_{l_i}}^\gamma\Big)<1\]
for every $p\io$. Put $A_\gamma'=\big\{k_{l_p}\colon\ p\io\big\}$. (The family of sets $A_\gamma'$'s will be used at the very end of the proof.)

Since $A_\gamma'\sub A_\gamma$ and $A_\gamma$
is an  infinite co-infinite subset of $\omega$, the closed subspace
$K_\delta$ of $2^\delta$ defined as
\[K_\delta=\Big(\bigcup_{p\io}V_{k_{l_p}}^\gamma\cup\big\{t_\gamma\big\}\Big)\times\{0\}\ \cup\ \Big(K_\gamma\sm\bigcup_{p\io}V_{k_{l_p}}^\gamma\Big)\times\{1\},\]
together with the naturally defined projection
$\pi_\gamma^\delta\colon K_\delta\to K_\gamma$  (being simply a
restriction of $\pr_\gamma^\delta$), is a simple extension of
$K_\gamma$ (see Lemma \ref{lem:simple_ext_convergent_antichain}). Note that
$\Big(\bigcup_{p\io}V_{k_{l_p}}^\gamma\cup\big\{t_\gamma\big\}\Big)\times\{0\}$
is a clopen subset of $K_\delta$. For $\alpha\in[\omega,\gamma)$ we define the mappings $\pi_\alpha^\delta\colon K_\delta\to K_\alpha$ by putting $\pi_\alpha^\delta=\pi_\alpha^\gamma\circ\pi_\gamma^\delta$. The $\delta$th-step is finished.

\medskip

We proceed in the above manner until we obtain the  limit
$K_{\omega_1}=\varprojlim\seq{K_\alpha}{\omega\le\alpha<\omega_1}$.
It follows directly from the construction that the inverse system
$\seq{K_\alpha}{\omega\le\alpha\le\omega_1}$ is based on simple
extensions and thus the Boolean algebra
$\aA=Clopen\big(K_{\omega_1}\big)$ is minimally generated.  We need
to show that $\aA$ has the Nikodym property, so, for the sake of
contradiction, let us assume that there exists an anti-Nikodym
sequence $\seqn{\mu_n}$ of measures on $\aA$.

\medskip

Again by virtue of Lemma \ref{lem:aN_antichain}, let
$\beta<\omega_1$ be such a minimal ordinal that for the antichain
$\seqk{W_k}$ of clopen subsets of $K_{\omega_1}$ defined for every
$k\io$ as $W_k=\big[F^\beta_k\big]_{\w_1}\cap K_{\omega_1}$ there exists an
increasing sequence $\seqk{n_k}$ satisfying for every $k\io$ the
following condition:
\[\tag{5}\big|\mu_{n_k}\big(W_k\big)\big|>2^k\cdot k.\]
Let $\xi<\omega_1$ be the minimal ordinal number such that  the
above inequality works for the subsequence $\seqk{\mu_{s_\xi(k)}\big(W_k\big)}$. Let
also $\zeta<\w_1$ be such that $F^\beta_k\in\finsub{\Skin{\zeta}}$ for
all $k\in\w$. 

For every $n\io$ and $F\in\finsub{\Skin{\w_1}}$ put
\[f(n)(F)=\mu_n\big([F]_{\omega_1}\cap K_{\omega_1}\big).\]
Since each $\mu_n$ is a finite measure, the sequence
$\vec{f}=\seqn{f(n)}$ belongs  to the set
$\big(\mathbb{R}^{[\Skin{\w_1}]^{<\w}}\big)^\omega$. For every
$n\io$ let $D_n$ be a club set in $\omega_1$ coming from Lemma
\ref{lem:var_det_club} for the function $f(n)$. Put
$D=\bigcap_{n\io}D_n$, so $D$ is a club in $\omega_1$. By the
$\Diamond$-properties of the sequence
$\seq{f_\alpha}{\omega\le\alpha<\omega_1}$, coming from Lemma
\ref{lem:diamond_clopens}, the set
\[\Gamma=\Big\{\gamma\in[\zeta,\omega_1)\cap D\colon\ (\forall n\io)\big(f(n)\rstr\finsub{\Skin{\gamma}}=f_\gamma(n)\big)\Big\}\]
is stationary in $\omega_1$. It follows that for all
$\gamma\in\Gamma$ and $n\io$ the function $\mu^\gamma_n$ defined in
$(ii)$ of the $(\gamma+1)$-th step of the construction (i.e., $\mu^\gamma_n\big([F]_\gamma\cap
K_\gamma\big)=f_\gamma(n)(F)$ for every $F\in
\finsub{\Skin{\gamma}}$)
  is a well-defined measure on
$Clopen\big(K_\gamma\big)$, since for every $F\in\finsub{\Skin{\gamma}}$ we have
\[\tag{7}\mu^\gamma_n\big([F]_\gamma\cap K_\gamma\big)=f_\gamma(n)(F)=f(n)(F)=\mu_n\big([F]_{\omega_1}\cap K_{\omega_1}\big).\]
Moreover, by the properties of the set $D$ (and  thus of $\Gamma$)
and the definition of the total variations of elements of
$\mathbb{R}^{[\Skin{\alpha}]^{<\w}}$ ($\alpha\le\omega_1$), for
every $\gamma\in\Gamma$, $n\io$ and $F\in\finsub{\Skin{\gamma}}$ we
have
\[\tag{8}\big|\mu^\gamma_n\big|\big([F]_\gamma\cap K_\gamma\big)=\big|f_\gamma(n)\big|_\gamma(F)=\]
\[=\big|f(n)\rstr\finsub{\Skin{\gamma}}\big|_\gamma(F)=|f(n)|_{\w_1}(F)=\big|\mu_n\big|\big([F]_{\omega_1}\cap K_{\omega_1}\big).\]
Note that for every $\gamma\in\Gamma$ and
$F\in\finsub{\Skin{\gamma}}$ we have
\[\big(\pr^{\w_1}_\gamma\big)^{-1}\big[[F]_\gamma\cap K_\gamma\big]=[F]_{\w_1}\cap
K_{\w_1},\]
and thus $(7)$ and $(8)$ yield
\[\tag{9} \mu^\gamma_n (V)=\mu_n \Big(\big(\pi^{\w_1}_\gamma\big)^{-1}[V]\Big)
\quad\text{and}\quad\big|\mu^\gamma_n\big|
(V)=\big|\mu_n\big|\Big(\big(\pi^{\w_1}_\gamma\big)^{-1}[V]\Big)\]
for all clopen $V\sub K_\gamma$ and $n\io$.

Let $\gamma\in\Gamma$. The sequence $\seqn{\mu^\gamma_n}$ is an
anti-Nikodym sequence  of measures on $Clopen(K_\gamma)$. Indeed,
the pointwise boundedness of $\seqn{\mu^\gamma_n}$ follows from (7),
and (5) implies the lack of the uniform boundedness (together with
the fact that $\gamma\ge\zeta$). Thus, $K_{\gamma+1}\neq
K_\gamma\times\{0\}$, i.e. at the $(\gamma+1)$-th step we conducted
the procedure of ``killing'' the anti-Nikodym sequence
$\seqn{\mu_n^\gamma}$. To do this, we used the antichain
$\seqk{W^\gamma_k}$, where $W^\gamma_k=\big[F^{\beta_\gamma}_k\big]_\gamma\cap K_\gamma$ for each $k\io$, and the
subsequence $\seqk{\mu^\gamma_{s_{\xi_\gamma}(k)}}$ such that $(2)$
holds. Moreover, $\beta_\gamma$ and afterwards $\xi_\gamma$ were
chosen to be minimal possible. Combining $\gamma\geq\zeta$,
$\big\{F^\beta_k\colon\ k\io\big\}\sub\finsub{\Skin{\zeta}}$, $(2)$, $(5)$, $(7)$,
as well as the fact that $\beta$ and then $\xi$ were chosen as the
minimal ordinals witnessing that $\seqn{\mu_n}$ is anti-Nikodym the way described in $(5)$, we conclude that
 $\beta=\beta_\gamma$ and $\xi=\xi_\gamma$.

Let again $\gamma\in\Gamma$. In what follows we shall use the notation from the description of the $(\gamma+1)$-th step of the construction, i.e. from the description of the space
$K_{\gamma+1}$. 
For every $p\io$, by $(3)$, $(4)$ and $(9)$ we have:
\[\tag{10}\big|\mu_{s_{\xi}(k_{l_p})}\Big(\big(\pi^{\w_1}_\gamma\big)^{-1}\Big[V_{k_{l_p}}^\gamma\Big]\Big)\big|>\sum_{i=0}^{p-1}\big|\mu_{s_{\xi}(k_{l_p})}\Big(\big(\pi^{\w_1}_\gamma\big)^{-1}\Big[V_{k_{l_i}}^\gamma\Big]\Big)\big|+p+1\]
and
\[\tag{11}\sum_{i>p}\Big|\mu_{s_{\xi}(k_{l_p})}\Big|\Big(\big(\pi^{\w_1}_\gamma\big)^{-1}\Big[V_{k_{l_i}}^\gamma\Big]\Big)<1.\]
Set $s_\gamma=\{\langle\gamma,0\rangle\}\in\Skin{\gamma+1}$, i.e. $\dom(s_\gamma)=\{\gamma\}$ and $\ran(s_\gamma)=\{0\}$, and note that for $F_\gamma=\big\{s_\gamma\big\}$ we have
\[\big[F_\gamma\big]_{\gamma+1}\cap K_{\gamma+1}=\Big(\bigcup_{p\io}V_{k_{l_p}}^\gamma\cup\big\{t_\gamma\big\}\Big)\times\{0\}\]
by the definition of $K_{\gamma+1}$.  Therefore, for every $p\io$ we
get
\[\big|\mu_{s_\xi(k_{l_p})}\big(\big[F_\gamma\big]_{\omega_1}\cap K_{\omega_1}\big)\big|=\big|\mu_{s_\xi(k_{l_p})}\Big(\big(\pi_{\gamma+1}^{\omega_1}\big)^{-1}\big[\big[F_\gamma\big]_{\gamma+1}\cap K_{\gamma+1}\big]\Big)\big|=\]

\smallskip

\[=\big|\mu_{s_\xi(k_{l_p})}\bigg(\big(\pi_{\gamma+1}^{\omega_1}\big)^{-1}\big[\Big(\bigcup_{i\io}V_{k_{l_i}}^\gamma\cup\big\{t_\gamma\big\}\Big)\times\{0\}\big]\bigg)\big|=\]

\smallskip

\[=\big|\mu_{s_\xi(k_{l_p})}\bigg(\big(\pi_{\gamma+1}^{\omega_1}\big)^{-1}\big[\Big(\bigcup_{i<p}V_{k_{l_i}}^\gamma\Big)\times\{0\}\big]\bigg)+\mu_{s_\xi(k_{l_p})}\Big(\big(\pi_{\gamma+1}^{\omega_1}\big)^{-1}\Big[V_{k_{l_p}}^\gamma\times\{0\}\Big]\Big)+\]
\[+\ \mu_{s_\xi(k_{l_p})}\bigg(\big(\pi_{\gamma+1}^{\omega_1}\big)^{-1}\big[\Big(\bigcup_{i>p}V_{k_{l_i}}^\gamma\cup\big\{t_\gamma\big\}\Big)\times\{0\}\big]\bigg)\big|=\]

\smallskip

\[=\big|\mu_{s_\xi(k_{l_p})}\Big(\big(\pi_{\gamma}^{\omega_1}\big)^{-1}\Big[\bigcup_{i<p}V_{k_{l_i}}^\gamma\Big]\Big)+\mu_{s_\xi(k_{l_p})}\Big(\big(\pi_{\gamma}^{\omega_1}\big)^{-1}\Big[V_{k_{l_p}}^\gamma\Big]\Big)+\]
\[+\ \mu_{s_\xi(k_{l_p})}\bigg(\big(\pi_{\gamma+1}^{\omega_1}\big)^{-1}\big[\Big(\bigcup_{i>p}V_{k_{l_i}}^\gamma\cup\big\{t_\gamma\big\}\Big)\times\{0\}\big]\bigg)\big|\ge\]

\smallskip

\[\ge\big|\mu_{s_\xi(k_{l_p})}\Big(\big(\pi^{\w_1}_\gamma\big)^{-1}\Big[V_{k_{l_p}}^\gamma\Big]\Big)\big|-\big|\mu_{s_\xi(k_{l_p})}\Big(\bigcup_{i<p}\big(\pi^{\w_1}_\gamma\big)^{-1}\Big[V_{k_{l_i}}^\gamma\Big]\Big)\big|-\]
\[-\big|\sum_{i>p}\mu_{s_\xi(k_{l_p})}\Big(\big(\pi^{\w_1}_\gamma\big)^{-1}\Big[V_{k_{l_i}}^\gamma\Big]\Big)\big|-
\big|\wh{\mu}_{s_\xi(k_{l_p})}
\Big(\big(\pi_{\gamma+1}^{\omega_1}\big)^{-1}\big(t_\gamma\concat 0
\big) \Big) \big|\ge\]

\smallskip

\[\ge\big|\mu_{s_\xi(k_{l_p})}\Big(\big(\pi^{\w_1}_\gamma\big)^{-1}\Big[V_{k_{l_p}}^\gamma\Big]\Big)\big|-\sum_{i<p}\big|\mu_{s_\xi(k_{l_p})}\Big(\big(\pi^{\w_1}_\gamma\big)^{-1}\Big[V_{k_{l_i}}^\gamma\Big]\Big)\big|-\]
\[-\sum_{i>p}\Big|\mu_{s_\xi(k_{l_p})}\Big|\Big(\big(\pi^{\w_1}_\gamma\big)^{-1}\Big[V_{k_{l_i}}^\gamma\Big]\Big)-
\big|\wh{\mu}_{s_\xi(k_{l_p})}
\Big(\big(\pi_{\gamma+1}^{\omega_1}\big)^{-1}\big(t_\gamma\concat 0
\big) \Big) \big|>\]

%
%
%

\smallskip

\[\tag{12}>p-\big|\wh{\mu}_{s_\xi(k_{l_p})}
\Big(\big(\pi_{\gamma+1}^{\omega_1}\big)^{-1}\big(t_\gamma\concat0
\big) \Big) \big|,\]
\noindent where the last inequality follows from $(10)$ and $(11)$. Note that by Lemma \ref{lem:simple_ext_empty_int} it holds that
\[\tag{13}\big(\pi_{\gamma+1}^{\omega_1}\big)^{-1}\big(t_\gamma\concat0
\big)\sub\partial_{K_{\w_1}}\Big(\bigcup_{p\in\w}\big(\pi^{\w_1}_\gamma\big)^{-1}\Big[V_{k_{l_p}}^\gamma\Big]\Big).\]

Since for every $\gamma\in\Gamma$ and $p\io$ we have $\big(\pi^{\w_1}_\gamma\big)^{-1}\Big[V_{k_{l_p}}^\gamma\Big]\sub W_{k_{l_p}}$ and
\[\bigvee\Big\{\big(\pi^{\w_1}_\gamma\big)^{-1}\big[V_i^\gamma\big]\colon i\in A_\gamma'\Big\}\]
exists in $Clopen\big(K_{\omega_1}\big)$ (it is equal to $\big[F_\gamma\big]_{\omega_1}\cap K_{\omega_1}$), by Lemma \ref{lem:measures_ad_families} there is $\eta<\omega_1$ such that for every $\gamma\in\Gamma$, $\gamma\ge\eta$, and $p\io$ we have:

\[\tag{14}\big|\wh{\mu}_{s_\xi(k_{l_p})}\big|\big(
\partial_{K_{\w_1}}\Big(\bigcup_{p\in\w}\big(\pi^{\w_1}_\gamma\big)^{-1}\Big[V_{k_{l_p}}^\gamma\Big]\Big) \big)=0,\]
which together with $(12)$ and $(13)$ yields that
 \[ \big|\mu_{s_\xi(k_{l_p})}\big(
\big[F_\gamma\big]_{\w_1}\cap K_{\w_1}\big)\big|>p  \]
for all $p\io$ and $\gamma\in\Gamma$ such that $\gamma\ge\eta$, thus
contradicting the pointwise boundedness of $\seqn{\mu_n}$. It follows that $\aA$ admits no anti-Nikodym sequences and hence has the Nikodym
property.
\end{proof}

\section{Consequences\label{sec:consequences}}

In this section we provide several consequences of Theorem \ref{thm:min_gen_nik}. They mostly serve to distinguish the Nikodym property from the Grothendieck property, especially from the point of view of measure-theoretic aspects of the Stone spaces of Boolean algebras having either of the two properties. In particular, their overall sense is that Boolean algebras with the Nikodym property may be \textit{less intricate} than Boolean algebras with the Grothendieck property.

\subsection{Uniform regular measures}

Babiker \cite{Bab77} introduced the following class of \textit{simple} probability measures on compact spaces.

\begin{definition}\label{def:unif_reg_meas}
A probability measure $\mu$ on a compact space $K$ is \textit{uniformly regular} if there exists a countable family $\cC$ of zero subsets of $K$ such that for every open subset $U$ of $K$ and every $\eps>0$ there is $F\in\cC$ such that $F\sub U$ and $\mu(U\sm F)<\eps$.
\end{definition}

The uniform regularity of measures was studied e.g. by Pol \cite{Pol82}, Mercourakis \cite{Mer96}, Borodulin-Nadzieja \cite{PBN07}, Krupski and Plebanek \cite{KP11}. In particular, Mercourakis \cite[Corollary 2.8]{Mer96} proved that every uniformly regular measure $\mu$ on a compact space $K$ admits a $\mu$-uniformly distributed sequence.

\begin{definition}\label{def:unif_distr_seq}
A probability measure $\mu$ on a compact space $K$ \textit{admits a uniformly distributed sequence} $\seqn{x_n\in K}$ if $\frac{1}{n}\sum_{i=0}^{n-1}\delta_{x_i}$ converges weakly* to $\mu$.
\end{definition}

D\v{z}amonja and Plebanek \cite[Lemma 4.1]{DP07} showed that if an inverse system of simple extensions has length at most $\omega_1$, then every non-atomic measure on its limit is uniformly regular. Their result together with Theorem \ref{thm:min_gen_nik} yields immediately the following corollary.

\begin{corollary}\label{cor:nik_unif_reg_meas}
Assuming $\Diamond$, there exists an infinite (minimally generated) Boolean algebra $\aA$ with the Nikodym property and such that every non-atomic probability measure on $St(\aA)$ is uniformly regular. In particular, every non-atomic probabiliy measure $\mu$ on $St(\aA)$ admits a $\mu$-uniformly distributed sequence.
\end{corollary}

Borodulin-Nadzieja \cite[Theorems 4.6 and 4.9]{PBN07} proved that the Stone space of a minimally generated Boolean algebra always carries a uniformly regular probability measure. On the other hand, we proved in \cite[Propositions 9.7 and 9.8]{KSZ20} that the existence of a uniformly regular probability measure on a compact space $K$ implies that $K$ does not have the Grothendieck property (since the space $C_p(K)$ has then the so-called \textit{Josefson--Nissenzweig property} which is equivalent for $K$ to fail an even weaker variant of the Grothendieck property, called \textit{the $\ell_1$-Grothendieck property}, see \cite[Theorem 6.7]{KSZ20}). Thus, Corollary \ref{cor:nik_unif_reg_meas} shows that the Nikodym property may be consistently different in this aspect from the Grothendieck property.

Note that the results discussed in the previous paragraph yield altogether that no minimally generated Boolean algebra may have the Grothendieck property (cf. \cite[Corollary 9.11]{KSZ20}).

\subsection{Maharam type}

Uniformly regular measures are a special case of measures having \textit{countable Maharam type},  defined as follows.

\begin{definition}\label{def:maharam_type}
\textit{The Maharam type} of a probability measure $\mu$ on a compact space $K$ is the minimal cardinality of a family $\cC$ of Borel subsets of $K$ such that for every Borel subset $B$ of $K$ and $\eps>0$ there exists $C\in\cC$ such that $\mu(B\triangle C)<\eps$.
\end{definition}

Equivalently, the Maharam type of a probability measure $\mu$ is the density of the Banach space $L_1(\mu)$ of all $\mu$-integrable real-valued functions on $K$. For more information on the topic, see Maharam \cite{Mah42}, Fremlin \cite{Fre89}, or Plebanek and Sobota \cite{PS14}. Note that measures of countable Maharam type are also called \textit{separable}.

Borodulin-Nadzieja \cite[Theorem 4.9]{PBN07} proved that the limit of an inverse system based on minimal extensions of compact spaces carries only measures of countable type, hence we immediately obtain the following corollary.

\begin{corollary}\label{cor:nik_sep_meas}
Assuming $\Diamond$, there exists an infinite Boolean algebra $\aA$ with the Nikodym property and such that every probability measure on $St(\aA)$ has countable Maharam type.
\end{corollary}

Corollary \ref{cor:nik_sep_meas} yields that carrying only measures of countable Maharam type is another feature that differs the Grothendieck property from the Nikodym property, since Krupski and Plebanek \cite[page 2189]{KP11} showed that if a Boolean algebra has the Grothendieck property, then its Stone space carries a measure of uncountable type (more precisely, there exists a measure of type not smaller than the so-called \textit{pseudo-interesection number} $\frakp$).

\subsection{Independent families}

Let us recall the following standard definition.

\begin{definition}\label{def:indep_families}
Let $\aA$ be a Boolean algebra. A family $\fF$ of elements of $\aA$ is \textit{independent} if and only if for every disjoint finite subsets $F,G\sub\fF$ we have:
\[\bigwedge_{A\in F}A\wedge\bigwedge_{A\in G}A^c\neq 0.\]
\end{definition}

It is an easy fact that every $\sigma$-complete Boolean algebra contains an independent family of size $\frakc$. Haydon \cite{Hay81} provided an argument (due to Argyros) proving that Boolean algebras with \textit{the Subsequential Completeness Property} (SCP), being a weakening of the $\sigma$-completeness, always contain uncountable independent families (in fact, of size at least $\frakp$). This result was later generalized by Koszmider and Shelah \cite{KS12} who showed that Boolean algebras having a very weak form of completeness called \textit{the Weak Subsequential Separation Property} (WSSP) must necessarily contain independent families of size $\frakc$. Both the $\sigma$-completeness and Haydon's SCP imply the WSSP as well as the Grothendieck property and the Nikodym property, however note that there exists a Boolean algebra with the WSSP, but with neither the Grothendieck property nor the Nikodym property.

Under the assumption of the Continuum Hypothesis, Talagrand \cite{Tal80} constructed a Boolean algebra $\aA$ with the Grothendieck property and such that every independent family in $\aA$ is at most countable. A natural question hence arises whether a similar example but with the Nikodym property instead of the Grothendieck property may be obtained. It appears that the Boolean algebra in Theorem \ref{thm:min_gen_nik} is such an example, since, by the result of Borodulin-Nadzieja \cite[Corollary 4.10]{PBN07}, no minimally generated Boolean algebra may contain an uncountable independent family.

\begin{corollary}\label{cor:nik_sep_meas}
Assuming $\Diamond$, there exists an infinite  Boolean algebra $\aA$ with the Nikodym property and such that every independent family $\fF\sub\aA$ is at most countable.
\end{corollary}

Note that several consistent examples of Boolean algebras with the Nikodym or Grothendieck property of size strictly less than $\frakc$ have been obtained, e.g. by Brech \cite{Bre06}, Sobota \cite{Sob19}, Sobota and Zdomskyy \cite{SZ19}; those algebras cannot of course contain independent families of size $\frakc$ and thus are far from being $\sigma$-complete.


\section{The second example}

Before we provide a consistent example of a minimally generated Boolean algebra whose Stone space is Efimov but the algebra itself does not have the Nikodym property, we need to establish an ambient family of Borel subsets of $\Cantor$ in which the construction will be conducted.

Let $\lambda$ denote the standard complete product measure on $\Cantor$. For each $n\io$ and $i\in\{0,1\}$ put $S_n^i=\{x\in\Cantor\colon\ x(n)=i\}$. Define the ambient family $\mathbb{A}$ as follows\footnote{The authors would like to thank Piotr Borodulin-Nadzieja and Omar Selim for bringing the sets $\mathbb{A}$ and $\mathbb{B}$ and Lemmas \ref{lem:ambient_not_nikodym} and \ref{lem:ambient_algebra} to their attention.}:

\[\mathbb{A}=\big\{A\in Bor(\Cantor)\colon\ \lim_{n\to\infty}n\big(\lambda\big(S_n^0\cap A\big)-\lambda\big(S_n^1\cap A\big)\big)=0\big\}.\]

$\mathbb{A}$ is not a Boolean algebra, but it of course contains $Clopen(\Cantor)$. The following immediate observation will be crucial for the proof of Theorem \ref{thm:min_gen_no_nik}.

\begin{lemma}\label{lem:ambient_not_nikodym}
Let $\aA$ be a Boolean subalgebra of $Bor(\Cantor)$ such that $Clopen(\Cantor)\sub\aA\sub\mathbb{A}$. Then, $\aA$ does not have the Nikodym property.
\end{lemma}
\begin{proof}
The sequence $\seqn{\mu_n}$ of measures on $\aA$ defined for each $n\io$ and $A\in\aA$ by the formula:
\[\mu_n(A)=n\big(\lambda\big(S_n^0\cap A\big)-\lambda\big(S_n^1\cap A\big)\big)\]
is an anti-Nikodym sequence. Indeed, for every $n\io$ we have $\big\|\mu_n\big\|=n$ and for every $A\in\aA$ and almost all $n\io$ it holds $\big|\mu_n(A)\big|<1$.
\end{proof}

For each $n\io$ put $\fF_n=\big\{[F]_\w\colon\ F\sub 2^n\big\}$, i.e. $\fF_n$ denotes the Boolean algebra of clopen subsets of $\Cantor$ generated by all the sets of the form $[\sigma]_\w$, where $\sigma\in 2^n$, and for every $A\in Bor(\Cantor)$ let:
\[\chi_n(A)=\min\big\{\lambda(A\triangle B)\colon\ B\in\fF_n\big\},\]
i.e. $\chi_n(A)$ measures how far a set $A$ from being a clopen generated by $2^n$ is. Let us also define:
\[\mathbb{B}=\big\{A\in Bor(\Cantor)\colon\ \lim_{n\to\infty}n\cdot\chi_n(A)=0\big\}.\]
The following lemma states basic properties of the family $\mathbb{B}$.

\begin{lemma}\label{lem:ambient_algebra}
$\mathbb{B}$ is a Boolean subalgebra of $Bor(\Cantor)$ and $Clopen(\Cantor)\sub\mathbb{B}\sub\mathbb{A}$.
\end{lemma}
\begin{proof}
We first prove that $\mathbb{B}$ is a Boolean algebra. Let $A,A'\in\mathbb{A}$. For every $n\io$ and $B,B'\in\fF_n$ we have:
\[\lambda(A\triangle B)=\lambda(A^c\triangle B^c),\]
so $n\cdot\chi_n(A)=n\cdot\chi_n(A^c)$ and hence $A^c\in\mathbb{B}$, and
\[(A\cup A')\triangle(B\cup B')\sub(A\triangle B)\cup(A'\triangle B'),\]
so
\[n\cdot\chi_n(A\cup A')\le n\cdot\chi_n(A)+n\cdot\chi_n(A')\]
and hence $A\cup A'\in\mathbb{B}$.

It is obvious that $Clopen(\Cantor)\sub\mathbb{B}$, so we only need to show that $\mathbb{B}\sub\mathbb{A}$. Let thus $A\in\mathbb{B}$ and fix $\eps>0$. There is $N\io$ such that $n\cdot\chi_n(A)<\eps/8$ for every $n>N$. Fix $n>N$ and $B\in\fF_n$ such that $n\cdot\lambda(A\triangle B)<\eps/8$. Since $B\in\fF_n$, it holds that:
\[\lambda\big(S_{n+1}^0\cap B\big)-\lambda\big(S_{n+1}^1\cap B\big)=0,\]
and so we have:
\[(n+1)\cdot\big|\lambda\big(S_{n+1}^0\cap A\big)-\lambda\big(S_{n+1}^1\cap A\big)\big|=\]
\[(n+1)\cdot\Big|\lambda\big(S_{n+1}^0\cap B\big)+\lambda\big(S_{n+1}^0\cap(A\sm B)\big)-\lambda\big(S_{n+1}^0\cap(B\sm A)\big)-\]
\[-\lambda\big(S_{n+1}^1\cap B\big)-\lambda\big(S_{n+1}^1\cap(A\sm B)\big)+\lambda\big(S_{n+1}^1\cap(B\sm A)\big)\Big|=\]
\[=(n+1)\cdot\Big|\lambda\big(S_{n+1}^0\cap(A\sm B)\big)-\lambda\big(S_{n+1}^0\cap(B\sm A)\big)-\lambda\big(S_{n+1}^1\cap(A\sm B)\big)+\lambda\big(S_{n+1}^1\cap(B\sm A)\big)\Big|\le\]
\[\le(n+1)\cdot4\cdot\lambda(A\triangle B)\le4n\cdot\lambda(A\triangle B)+4\lambda(A\triangle B)<\eps/2+\eps/(2n)\le\eps.\]
In other words, there is $N\io$ such that for every $n>N$ we have:
\[\big|n\big(\lambda\big(S_n^0\cap A\big)-\lambda\big(S_n^1\cap A\big)\big)\big|<\eps.\]
It follows that
\[\lim_{n\to\infty}n\big(\lambda\big(S_n^0\cap A\big)-\lambda\big(S_n^1\cap A\big)\big)=0\]
and thus $A\in\mathbb{A}$.

\end{proof}

Lemma \ref{lem:ambient_not_nikodym} implies that $\mathbb{B}$ does not have the Nikodym property. The next two lemmas will constitute the heart of the proof of Theorem \ref{thm:min_gen_no_nik}.

\begin{lemma}\label{lem:intersection_in_B}
Let $k\ge k'>0$ be natural numbers, $\eps>0$, $A\in\fF_{k'}$, and $V\in\mathbb{B}$ be such that $k\cdot\chi_k(V)<\eps$. Then, $k\cdot\chi_k(V\cap A)<\eps$.
\end{lemma}
\begin{proof}
Let $B\in\fF_k$ be such that $k\cdot\lambda(V\triangle B)<\eps$. Since for any three sets $X,Y$ and $Z$ we have $(X\cap Z)\triangle(Y\cap Z)=(X\triangle Y)\cap Z$, we immediately get that:
\[k\cdot\lambda\big((V\cap A)\triangle(B\cap A)\big)=k\cdot\lambda\big((V\triangle B)\cap A\big)\le k\cdot\lambda(V\triangle B)<\eps.\]
Since $A\in\fF_k$, $A\cap B\in\fF_k$, and thus $k\cdot\chi_k(V\cap A)<\eps$.
\end{proof}

\begin{lemma}\label{lem:B_antichain_supremum}
Let $\bB$ be a subalgebra of $\mathbb{B}$ containing $Clopen(\Cantor)$. Let $\seqn{q_n}$ be a sequence in $St(\bB)$ and $\seqn{V_n}$ an antichain in $\bB$ such that $V_n\in q_n$ for every $n\io$. Then, there is a sequence $\seqn{\sigma_n\in 2^{<\omega}}$ such that
\[\bigcup_{n\io}\big(V_n\cap\big[\sigma_n\big]_\omega\big)\in\mathbb{B}\]
and $V_n\cap\big[\sigma_n\big]_\omega\in q_n$ for every $n\io$.
\end{lemma}
\begin{proof}
We need to find two sequences $\seqn{k_n\io}$ and $\seqn{\sigma_n\in 2^{k_n}}$ such that
\[\lim_{l\to\infty}l\cdot\chi_l\Big(\bigcup_{n\io}\big(V_n\cap\big[\sigma_n\big]_\omega\big)\Big)=0,\]
and $V_n\cap\big[\sigma_n\big]_\omega\in q_n$ for every $n\io$. First, note that for every $n\io$ and $\sigma\in 2^{<\omega}$ we have $V_n\cap[\sigma]_\omega\in\bB$ (since $Clopen(\Cantor)\sub\bB$ and $\bB$ is a Boolean algebra) and that for every $m\io$ there is $\sigma\in 2^m$ such that $V_n\cap[\sigma]_\omega\in q_n$ (since  $q_n$ is an ultrafilter in $\bB$ and $\bigcup_{\sigma\in 2^m}[\sigma]_\omega=\Cantor$ is the unit element of $\bB$).

We construct the sequences inductively. Let $k_0>0$ be arbitrary and pick $\sigma_0\in 2^{k_0}$ such that $V_0\cap\big[\sigma_0\big]_\omega\in p_0$. Put $W_0=V_0\cap\big[\sigma_0\big]_\omega$.

Since $W_0,V_1\in\bB\sub\mathbb{B}$, there is $k_1>k_0$ such that for every $k\ge k_1$ we have:
\[k\cdot\chi_k\big(W_0\big)<1/2^1\quad\text{and}\quad k\cdot\chi_k\big(V_1\big)<1/2^1.\]
Pick $\sigma_1\in 2^{k_1}$ such that $V_1\cap\big[\sigma_1\big]_\omega\in p_1$ and put $W_1=V_1\cap\big[\sigma_1\big]_\omega$. By Lemma \ref{lem:intersection_in_B}, for every $k\ge k_1$ it holds:
\[k\cdot\chi_k\big(W_1\big)<1/2^1.\]

Again, since $W_0,W_1,V_2\in\bB\sub\mathbb{B}$, there is $k_2>k_1$ such that for every $k\ge k_2$ we have:
\[k\cdot\chi_k\big(W_0\big)<1/2^2\quad\text{,}\quad k\cdot\chi_k\big(W_1\big)<1/2^2\quad\text{and}\quad k\cdot\chi_k\big(V_2\big)<1/2^2.\]
Pick $\sigma_2\in 2^{k_2}$ such that $V_2\cap\big[\sigma_2\big]_\omega\in p_2$ and put $W_2=V_2\cap\big[\sigma_2\big]_\omega$. By Lemma \ref{lem:intersection_in_B}, for every $k\ge k_2$ it holds:
\[k\cdot\chi_k\big(W_2\big)<1/2^2.\]

Continue in this manner until you get a strictly increasing sequence $\seqn{k_n}$ of natural numbers and an infinite sequence $\seqn{\sigma_n\in 2^{k_n}}$ such that for every $n>0$, $k\ge k_n$ and $i\le n$ we have:
\[\tag{$*$}k\cdot\chi_k\big(W_i\big)<1/2^n,\]
where $W_i=V_i\cap\big[\sigma_i\big]_\omega$ and $W_i\in p_i$. Put:
\[W=\bigcup_{n\io}W_n.\]
We claim that $W\in\mathbb{B}$. Indeed, by ($*$), for every $n\io$ and $k\in k_{n+1}\setminus k_n$ there are subsets $S_0,\ldots,S_n\sub 2^k$ such that for every $i\le n$ it holds:
\[\tag{$**$}k\cdot\lambda\big(W_i\triangle\big[S_i\big]_\w\big)<1/2^n.\]
Since $\seqn{W_n}$ is an antichain, we get that:
\[k\cdot\lambda\Big(W\triangle\bigcup_{i\le n}\big[S_i\big]_\w\Big)=k\cdot\lambda\Big(\big(\bigcup_{j\le n}W_j\cup\bigcup_{j> n}W_j\big)\triangle\bigcup_{i\le n}\big[S_i\big]_\w\Big)=\]
\[=k\cdot\lambda\Big(\big(\bigcup_{j\le n}W_j\cup\bigcup_{j> n}W_j\big)\sm\bigcup_{i\le n}\big[S_i\big]_\w\Big)+k\cdot\lambda\Big(\bigcup_{i\le  n}\big[S_i\big]_\w\sm\big(\bigcup_{j\le n}W_j\cup\bigcup_{j>n}W_j\big)\Big)\le\]
\[\le k\cdot\lambda\Big(\bigcup_{j\le n}W_j\sm\bigcup_{i\le  n}\big[S_i\big]_\w\Big)+k\cdot\lambda\Big(\bigcup_{j> n}W_j\Big)+k\cdot\lambda\Big(\bigcup_{i\le n}\big[S_i\big]_\w\sm\bigcup_{j\le n}W_j\Big)=\]
\[=k\cdot\lambda\Big(\bigcup_{j\le n}W_j\triangle\bigcup_{i\le n}\big[S_i\big]_\w\Big)+k\cdot\lambda\Big(\bigcup_{j\ge n+1}W_j\Big)\le\]
\[\le k\cdot\sum_{i\le n}\lambda\big(W_i\triangle\big[S_i\big]_\w\big)+k\cdot\sum_{j\ge n+1}1/2^{k_j}<\]
\[<(n+1)/2^n+2k_{n+1}/2^{k_{n+1}},\]
where the last inequality follows from ($**$) and the fact that $\lambda\big(W_j)\le1/2^{k_j}$ for every $j\io$. Since the values in the last line converge to $0$, we obtain that $\lim_{k\to\infty}k\cdot\chi_k(W)=0$, too.
\end{proof}

\begin{theorem}\label{thm:min_gen_no_nik}
Assuming $\Diamond$, there exists a minimally generated Boolean algebra $\bB$ such that $Clopen(\Cantor)\sub\bB\sub\mathbb{B}$ and the Stone space $St(\bB)$ is an Efimov space homeomorphic to a perfect subset of $2^{\omega_1}$. In particular, $\bB$ does not have the Nikodym property.
\end{theorem}
\begin{proof}
As one can suspect, the construction will be similar to the one from the proof of Theorem \ref{thm:min_gen_nik}, but much easier. Let $\seq{f_\alpha\in\big(2^\alpha)^\omega}{\omega\le\alpha<\omega_1}$ be a sequence given by $\Diamond$. 
We will construct inductively the following three objects:
\begin{enumerate}[(a)]
    \item an increasing sequence $\seq{\bB_\alpha}{\omega\le\alpha\le\omega_1}$ of countable subalgebras of the Boolean algebra $\mathbb{B}$ such that $\bB_0=Clopen(\Cantor)$, $\bB_{\alpha+1}$ will be a minimal extension of $B_\alpha$ for every $\alpha\in[\omega,\omega_1)$, and $\bB_\lambda=\bigcup_{\omega\le\alpha<\lambda}\bB_\alpha$ for every limit ordinal $\lambda\in(\omega,\omega_1]$;
    \item an inverse system $\seq{K_\alpha,\pi_\alpha^\beta}{\omega\le\alpha<\beta\le\omega_1}$ based on simple extensions and such that $K_\alpha$ will be a closed
    perfect subset of the space $2^\alpha$ for every
    $\omega\le\alpha\le\omega_1$, $\pi_\alpha^\beta\colon K_\beta\to
    K_\alpha$ will be such that $\pi_\alpha^\beta=\pr_\alpha^\beta\rstr
    K_\beta$ for every $\omega\le\alpha<\beta\le\omega_1$, and $K_{\omega_1}$ will be an Efimov space;
    \item a sequence $\seq{\varphi_\alpha\colon K_\alpha\to St\big(\bB_\alpha\big)}{\omega\le\alpha\le\omega_1}$ of homeomorphisms satisfying for every $\omega\le\alpha<\beta\le\omega_1$ the following inequality:
    \[\tag{$*$}\rho_\alpha^\beta\circ\varphi_\beta=\varphi_\alpha\circ\pi_\alpha^\beta,\]
    where $\rho_\alpha^\beta$ denotes the restriction of ultrafilters from the algebra $\bB_\beta$ to its subalgebra $\bB_\alpha$, i.e. $\rho_\alpha^\beta(x)=x\cap\bB_\alpha$ for each $x\in St\big(\bB_\beta\big)$.
\end{enumerate}
We start with $\bB_\omega=Clopen(\Cantor)$ and $K_\omega=2^\omega$, and let $\varphi_\omega\colon K_\w\to St\big(\bB_\w\big)$ be an arbitrary homeomorphism. Let us
assume that for some $\omega<\delta<\omega_1$ we have already
constructed initial segments of the demanded sequences as described in Conditions (a)--(b): $\seq{\mathcal B_\alpha}{\omega\le\alpha<\delta}$, $\seq{K_\alpha,\pi_\alpha^\beta}{\omega\le\alpha<\beta<\delta}$, and $\seq{\varphi_\alpha}{\omega\le\alpha<\delta}$.
If $\delta$ is a limit ordinal, then simply put $\bB_\delta=\bigcup_{\omega\le\alpha<\delta}\bB_\alpha$ and
$K_\delta=\varprojlim\seq{K_\alpha}{\omega\le\alpha<\delta}$. By \cite[Prop. 2.5.10]{Eng89}, the limit mapping $\varphi_\delta=\varprojlim\seq{\varphi_\alpha}{\omega\le\alpha<\delta}$, $\varphi_\delta\colon K_\delta\to St\big(\bB_\delta\big)$, is a homeomorphism and by \cite[Eq. (6) on Page 101]{Eng89} it satisfies the equation ($*$) for every $\alpha\le\delta$, so we are done. If
$\delta=\gamma+1$ for some ordinal $\gamma$, then we proceed as
follows.
\medskip

We let $\bB_{\gamma+1}=\bB_\gamma$, $K_{\gamma+1}=K_\gamma\times\{0\}$, $\pi_\gamma^{\gamma+1}=\pr_\gamma^{\gamma+1}\rstr K_{\gamma+1}$, and $\varphi_{\gamma+1}=\varphi_\gamma\circ\pi_\gamma^{\gamma+1}$, unless the sequence $f_\gamma=\seqn{f_\gamma(n)\in 2^\gamma}$, given by $\Diamond$, is a non-trivial convergent sequence in the space $K_\gamma$ (which, recall, is a closed subset of $2^\gamma$). If this is the case, then we shall make sure that $f_\gamma$ cannot be extended to a non-trivial convergent sequence in $K_{\omega_1}$ as described below. First, let $p_\gamma=\lim_{n\to\infty}f_\gamma(n)$, $q_\gamma=\varphi_\gamma\big(p_\gamma\big)$ and  $g_\gamma=\seqn{\varphi_\gamma\big(f_\gamma(n)\big)}$---as $\varphi_\gamma$ is a homeomorphism, $g_\gamma$ is a well-defined non-trivial convergent sequence in the Stone space $St\big(\bB_\gamma\big)$ and $q_\gamma=\lim_{n\to\infty}g_\gamma(n)$. Let $\seqn{U_n^\gamma}$ be an antichain in $\bB_\gamma$ such that $g_\gamma(n)\in U_n^\gamma$ for every $n\io$. Since $\bB_\gamma$ is countable and thus $St\big(\bB_\gamma\big)$ is metrizable, we may find a sequence $\seqn{V_n^\gamma}$ in $\bB_\gamma$ convergent to $q_\gamma$ and such that $g_\gamma(n)\in V_n^\gamma\sub U_n^\gamma$ for every $n\io$ (cf. the case of the sequence $\seql{V_{k_l}^\gamma}$ converging to $t_\gamma$ in the proof of Theorem \ref{thm:min_gen_nik}).

For every $n\in\omega\sm\E$ put $W_n^\gamma=V_n^\gamma$. Since $Clopen(\Cantor)$ is a subalgebra of $\bB$, Lemma \ref{lem:B_antichain_supremum} implies that there is an antichain $\seq{W_n^\gamma}{n\in\E}$ in $\bB$ such that $g_\gamma(n)\in W_n^\gamma\sub V_n^\gamma$ for every $n\in\E$ and that the supremum $W_\gamma=\bigvee_{n\in\E}W_n^\gamma$ (being just the union $\bigcup_{n\in\E}W_n^\gamma$) exists in $\mathbb{B}$. Naturally, the antichain $\seqn{W_n^\gamma}$ (of clopen subsets of $St\big(\bB_\gamma\big)$) converges to $q_\gamma$, too.

Let $\bB_\delta$ be a subalgebra of $\mathbb{B}$ generated by $\bB_\gamma$ and the supremum $W_\gamma$. Of course, $\bB_\delta$ is countable. We now apply Lemma \ref{lem:min_ext_convergent_antichain} (with $\aA=\bB_\gamma$, $\bB=\bB_\delta$, $V_n=W_n^\gamma$ ($n\io$), $t=q_\gamma$, $A=\E$, $\cC=\mathbb{B}$, $K=K_\gamma$, $\varphi=\varphi_\gamma$). Part (1) of the lemma yields that $\bB_\delta$ is a minimal extension of $\bB_\gamma$. By Part (2) of the lemma we get a homeomorphism $\varphi_\delta\colon K_\delta\to St\big(\bB_\delta\big)$, where $K_\delta$ is a perfect subset of $2^\delta$ defined as:
\[K_\delta=\Big(\bigcup_{n\in \E}\varphi_\gamma^{-1}\big[W_n^\gamma\big]\cup\big\{p_\gamma\big\}\Big)\times\{0\}\ \cup\ \Big(K_\gamma\sm\bigcup_{n\in \E}\varphi_\gamma^{-1}\big[W_n^\gamma\big]\Big)\times\{1\},\]
such that
\[\rho_\gamma^\delta\circ\varphi_\delta=\varphi_\gamma\circ\pi_\gamma^\delta,\]
where $\pi_\gamma^\delta\colon K_\delta\to K_\gamma$ is simply a
restriction of $\pr_\gamma^\delta$ to $K_\delta$, so ($*$) is satisfied for the pair $\gamma<\delta$ and thus for every pair $\alpha<\delta$. It also holds:
\[\varphi_\delta^{-1}\big[W_\gamma\big]=\Big(\bigcup_{n\in \E}\varphi_\gamma^{-1}\big[W_n^\gamma\big]\cup\big\{p_\gamma\big\}\Big)\times\{0\}.\]
By Lemma \ref{lem:simple_ext_convergent_antichain}, the space $K_\delta$ together with the mapping $\pi_\gamma^\delta\colon K_\delta\to K_\gamma$ is a simple extension of the space $K_\gamma$. For $\alpha\in[\omega,\gamma)$ we define the mappings $\pi_\alpha^\delta\colon K_\delta\to K_\alpha$ by putting $\pi_\alpha^\delta=\pi_\alpha^\gamma\circ\pi_\gamma^\delta$. The $\delta$-th step of the construction is thus finished.

\medskip

We proceed in the above manner until we obtain a sequence $\seq{\bB_\alpha}{\omega\le\alpha\le\omega_1}$ of Boolean algebras as in Condition (a), an inverse system $\seq{K_\alpha,\pi_\alpha^\beta}{\omega\le\alpha\le\beta\le\omega_1}$ of compact spaces as in Condition (b), and a sequence $\seq{\varphi_\alpha\colon K_\alpha\to St\big(\bB_\alpha\big)}{\omega\le\alpha\le\omega_1}$ of homeomorphisms as in Condition (c). Put $\bB=\bB_{\w_1}$. It follows that $\bB$ is minimally generated and Lemma \ref{lem:ambient_not_nikodym} yields that $\bB$ does not have the Nikodym property, as $Clopen(\Cantor)\sub\bB\sub\mathbb{B}$. We still need to show that $St(\bB)$ is an Efimov space so, for the sake of contradiction, let us suppose that there is a non-trivial convergent sequence inside $St(\bB)$ (note that there are no copies of $\bo$ since $\bB$ is minimally generated). Since $St(\bB)$ and $K_{\omega_1}$ are homeomorphic (via the mapping $\varphi_{\omega_1}$), there is a non-trivial convergent sequence $f=\seqn{f(n)}$ in $K_{\omega_1}$. Let $p=\lim_{n\to\infty}f(n)$.

\medskip

There exists $\xi\in[\omega,\omega_1)$ such that for every $\alpha\in[\xi,\w_1)$ and every $n\neq m\io$ we have $f(n)\rstr\alpha\neq p\rstr\alpha$ and $f(n)\rstr\alpha\neq f(m)\rstr\alpha$---it follows that for such $\alpha$ the sequence $\seqn{f(n)\rstr\alpha}$ converges to $p\rstr\alpha$ in $K_\alpha$. By the $\Diamond$-properties of the sequence $\seq{f_\alpha}{\omega\le\alpha\le\omega_1}$, the set
\[\Gamma=\Big\{\gamma\in[\xi,\omega_1)\colon\ (\forall n\io)\big(f(n)\rstr\gamma=f_\gamma(n)\big)\Big\}\]
is stationary in $\omega_1$.

Fix $\gamma\in\Gamma$. Since $f(n)\rstr\gamma=f_\gamma(n)$ for every $n\io$, the sequence $f_\gamma$ is a non-trivial convergent sequence in the space $K_\gamma$ with the limit $p\rstr\gamma$. Consequently, during the $\gamma$-th step of the construction we extended the space $K_\gamma$ to the space $K_{\gamma+1}$ with the new clopen set $\varphi_{\gamma+1}^{-1}\big[W_\gamma\big]$, so $K_{\gamma+1}\neq K_\gamma\times\{0\}$. Obviously, $p\rstr\gamma=p_\gamma$, and $p_\gamma$ was the only point in $K_\gamma$ that was split to two distinct points in $K_{\gamma+1}$. It follows that for every $n\io$ the element $f_\gamma(n)$ of $K_\gamma$ was extended to the unique element $\big(\pi_\gamma^{\gamma+1}\big)^{-1}\big(f_\gamma(n)\big)$ of $K_{\gamma+1}$. 
For every $n\io$ the following also holds:
\[\big(\pi_\gamma^{\gamma+1}\big)^{-1}\big(f_\gamma(n)\big)\in\big(\pi_\gamma^{\gamma+1}\big)^{-1}\Big[\varphi_{\gamma}^{-1}\big[W_n^\gamma\big]\Big]=\varphi_{\gamma+1}^{-1}\big[W_n^\gamma\big].\]
For every $n\in\E$ it holds $W_n^\gamma\sub W_\gamma$ and for every $n\in\omega\sm\E$ we have $W_n^\gamma\cap W_\gamma=\emptyset$, so
\[\big(\pi_\gamma^{\gamma+1}\big)^{-1}\big(f_\gamma(n)\big)\in\varphi_{\gamma+1}^{-1}\big[W_\gamma\big]\]
for every $n\in\E$, and
\[\big(\pi_\gamma^{\gamma+1}\big)^{-1}\big(f_\gamma(n)\big)\not\in\varphi_{\gamma+1}^{-1}\big[W_\gamma\big]\]
for every $n\in\omega\sm\E$. As $\varphi_{\gamma+1}^{-1}\big[W_\gamma\big]$ is clopen in $K_{\gamma+1}$, this implies that the sequence $\seqn{\big(\pi_\gamma^{\gamma+1}\big)^{-1}\big(f_\gamma(n)\big)}$ is not convergent in $K_{\gamma+1}$. But for every $n\io$ we have:
\[f(n)\rstr(\gamma+1)=\big(\pi_\gamma^{\gamma+1}\big)^{-1}\big(f(n)\rstr\gamma\big)=\big(\pi_\gamma^{\gamma+1}\big)^{-1}\big(f_\gamma(n)\big)\]
and $\gamma\ge\xi$, so $\seqn{\big(\pi_\gamma^{\gamma+1}\big)^{-1}\big(f_\gamma(n)\big)}$ is convergent in $K_{\gamma+1}$, which is a contradiction. This proves that the Stone space $St(\bB)$ does not have any non-trivial convergent sequences and thus is an Efimov space.
\end{proof}

\section{Open questions}

We now provide several open problems. In the context of Theorems \ref{thm:min_gen_nik} and \ref{thm:min_gen_no_nik} (or the results presented in \cite{DPM09} and \cite{DS13}), the most natural question is obviously the following one.

\begin{question}\label{ques:martin_nik}
Can we exchange $\Diamond$ in the assumption of Theorem \ref{thm:min_gen_nik} or Theorem
 \ref{thm:min_gen_no_nik}
 by the Continuum Hypothesis or Martin's axiom?
\end{question}

The next two questions concern consequences of Theorem \ref{thm:min_gen_nik} presented in Section \ref{sec:consequences}.

\begin{question}
Assuming Martin's axiom, does there exist an infinite Boolean algebra with the Nikodym property whose Stone space carries only measures having countable Maharam type?
\end{question}

\begin{question}
Assuming Martin's axiom, does there exist an infinite Boolean algebra with the Nikodym property and such that its every independent subfamily is at most countable?
\end{question}

Dow and Shelah \cite{DS13} proved that Martin's axiom suffices to construct an Efimov space. The space they obtain is the limit of an inverse system based on simple extensions and thus does not have the Grothendieck property, however we do not know whether its Boolean algebra of clopen subsets has the Nikodym property, cf. Question \ref{ques:martin_nik}. In fact, it even seems to be unknown whether Martin's axiom implies the existence of a Boolean algebra with the Nikodym property or the Grothendieck property (or both) whose Stone space is Efimov.

\begin{question}
Assuming Martin's axiom, does there exist an infinite Boolean algebra with the Nikodym property or the Grothendieck property whose Stone space is an Efimov space?
\end{question}

Recall the following variation of the Nikodym property.

\begin{definition}
A Boolean algebra $\aA$ has \textit{the strong Nikodym property} if for every increasing sequence $\seqn{A_n}$ of subsets of $\aA$ for which $\aA=\bigcup_{n\io}A_n$ there exists $n_0\io$ having the property that if $\seqn{\mu_n}$ is a sequence of measures on $\aA$ such that $\sup_{n\io}\big|\mu_n(a)\big|<\infty$ for every $a\in A_{n_0}$, then $\sup_{n\io}\big\|\mu_n\big\|<\infty$.
\end{definition}

Valdivia \cite{ValOld} proved that every $\sigma$-complete Boolean algebra has the strong Nikodym property (cf. also \cite{LoPel}, \cite{KaLoPel} and \cite{LoAlfMaMo}), however it is unknown whether every Boolean algebra with the Nikodym property has the strong Nikodym property, too (see Valdivia \cite[Problem 1]{Val}).

\begin{question}
Assuming $\Diamond$, does there exist a Boolean algebra with the Nikodym property but without the strong Nikodym property?
\end{question}
%

\begin{question}
If $\aA$ is a minimally generated Boolean algebra with the Nikodym property, does $\aA$ have the strong Nikodym property? In particular, does the algebra $\aA$ constructed in the proof of Theorem \ref{thm:min_gen_nik} have the strong Nikodym property?
\end{question}

\end{document}